\documentclass[12pt]{article}
\usepackage{amsthm,enumerate}
\usepackage{amssymb}
\usepackage{amsfonts, mathrsfs}
\usepackage{amsmath}
\usepackage[notcite,notref]{showkeys} % shows labels
\usepackage{hyperref}
\usepackage{color}\usepackage{graphicx}
\allowdisplaybreaks

\usepackage{tikz}
\usetikzlibrary{arrows}
\usetikzlibrary{decorations}
\usetikzlibrary{shapes.misc}
\newcommand{\pic}[1]{
\begin{tikzpicture}
#1
\end{tikzpicture}}

\newcommand{\pr}{\mathbf{Pr}}
\newcommand{\E}[1]{{\mathbf E}\left[#1\right]}
\newcommand{\e}{{\mathbf E}}

\newcommand{\bag}{\begin{align}}
\newcommand{\bags}{\begin{align*}}
\newcommand{\eag}{\end{align*}}
\newcommand{\eags}{\end{align*}}

\newtheorem{thm}{Theorem}
\newtheorem{lem}[thm]{Lemma}

\newtheorem{cor}[thm]{Corollary}

\newcommand\cC{\mathcal C}
\newcommand\cD{\mathcal D}

\newcommand\cF{\mathcal F}

\newcommand{\brac}[1]{\left(#1\right)}
\newcommand{\bfrac}[2]{\brac{\frac{#1}{#2}}}

\newcommand{\bd}{\textbf d}
\newcommand{\G}{\Gamma}
\newcommand{\bI}{\mathbb{I}}

\newcommand{\f}{\phi}
\newcommand{\ca}{\alpha}

\author{Michael Anastos} 
\title{Finding perfect matchings in random regular graphs in linear time}

\begin{document}
\maketitle

\begin{abstract}
In a seminal paper on finding large matchings in sparse random graphs, Karp and Sipser \cite{KS} proposed two algorithms for this task. The second algorithm has been intensely studied, but due to technical difficulties, the first algorithm has received less attention. Empirical results in \cite{KS} suggest that the first algorithm is superior. In this paper we analyze an adapted version of the first algorithm, the \textsc{Reduce-Construct} algorithm.
The \textsc{Reduce-Construct} algorithm is proposed in \cite{AF} and it is shown that  it finds a maximum matching in random $k$-regular graphs in linear time in expectation for $k\in \{3,4\}$. 
We extend the analysis done in \cite{AF} to random $k=O(1)$-regular graphs. We show that \textsc{Reduce-Construct} finds a maximum matching in such a graph in linear time in expectation, as opposed to $O(n^{3/2})$ time for the worst-case. 
\end{abstract}
\section{Introduction}
Given a graph $G=(V,E)$, a matching $M$ of $G$ is a subset of edges such that no vertex is incident to two edges in $M$. Finding a maximum cardinality matching is a central problem in algorithmic graph theory. The most efficient algorithm for general graphs is that given by Micali and Vazirani  \cite{MV} and runs in $O(|E||V|^{1/2})$ time.

In this paper we analyze the \textsc{Reduce-Construct} algorithm for finding perfect matchings in random regular graphs.
The \textsc{Reduce-Construct} algorithm is proposed in \cite{AF}
and it is an adaptation of Algorithm 1, as it is stated in \cite{KS},
given by Karp and Sipser. Algorithm 1 was proposed for finding a large matching in the random graph $G_{n,m},m=cn/2$ for some positive constant $c>0$. 

The \textsc{Reduce-Construct} algorithm 
can be split into two algorithms, the \textsc{Reduce} algorithm and the 
\textsc{Construct} algorithm.
\textsc{Reduce} sequentially reduces the graph until it reaches the empty graph. Then  \textsc{Construct}  unwinds some of the actions that \textsc{Reduce} has taken and grows a matching which is then output. 
For the complete description of the algorithm see \cite{AF}. In \cite{AF}
Anastos and Frieze proved that  \textsc{Reduce-Construct}
finds a maximum matching in  random $k$-regular graphs in linear time in expectation, for $k\in \{3,4\}$. In this paper we extend this result 
to $(\ca,3,k)$-\emph{dominant} random graphs of minimum degree 3 and maximum $k$, $\ca=1.17$. For a graph $G$ let $n_i(G)$ be the number of vertices of degree $i$. Then we define the set of $(\ca,3,k)$-\emph{dominant} random graphs $\cC_{3,k}$ by
$$\cC_{3,k}:=\{G: \cD_{k,j}(G) \text{ holds  for all } 3<j\leq k\}$$
where 
$$\cD_{k,j}(G):=\{ n_j(G) \geq \ca n_{j-1}(G) - (\log^2 n -k)n^{0.8}/2^j\}.$$
We discuss the role of $\cC_{3,k}$  in Subsection 2.3.
We proceed by stating the main Theorem of this paper.
Given a degree sequence $\bd$,  we can  generate a random (multi-)graph $G = G([n],E)$ with  degree sequence $\bd$ using the configuration model of Bollob\'as \cite{Bol}.
\begin{thm}\label{main}
Let $3\leq k=O(1)$.
Let $G\in \cC_{3,k}$ be a random (multi)-graph with degree sequence $\bd$,   and no loops. 
Then with probability $1-o(n^{-0.5})$ \textsc{Reduce-Construct} finds a (near) perfect matching in $O(n)$ time.
\end{thm}
A (near) perfect matching is one of size $\lfloor n/2 \rfloor$.
The probability in Theorem  \ref{main} is taken over both the randomness of the algorithm and the randomness  of the graph. Thus it also takes into account the probability that $G$ does not have a (near) perfect matching. In such an event inherently \textsc{Reduce-Construct} cannot find one. The following corollary is a direct consequence of Theorem \ref{main}.
\begin{cor}\label{mm}
Let $3\leq k=O(1)$. Let $G\in \cC_{3,k}$ be a random (multi)-graph with degree sequence $\bd$,   and no loops. Then there exists an algorithm that  finds a maximum matching of $G$ in $O(n)$ time in expectation.
\end{cor}
\begin{proof}
We first implement \textsc{Reduce-Construct} to find a perfect mathing of $G$ in $O(n)$ time. Theorem \ref{main} states that it fails to do so with probability $1-o(n^{-0.5})$. In such an event we impliment the Micali-Vazirani   algorithm \cite{MV} to find a maximum matching in
 $O(|E||V|^{1/2})=O(n^{1.5})$ time.
\end{proof}
A random $k$-regular (multi)-graph, $3\leq k=O(1)$, with no loops is simple 
with probability O(1). Furthermore it is $(\ca,3,k)$-dominant, $\ca=1.17$. Therefore both, Theorem \ref{main} and Corollary \ref{mm} hold in the special case where $G$ is a random $k$-regular graph.
\vspace{3mm}
\\For the rest of this paper we use the abbreviation w.h.p.\@ (with high probability) to mean with probability $1-o(n^{-0.5})$.
\section{Proof's Mechanics}
At various points we are going to use results proved in \cite{AF}. 
In this section we state those results. The reader is strongly advised to read Sections 1 \& 2 of \cite{AF} for better understanding.
We start by describing \textsc{Reduce}. 
We assume that our input (multi-)graph $G = G([n],E)$ has degree sequence $\bd$ and is generated by the configuration model of Bollob\'as \cite{Bol}. Let $W=[2\nu]$, $2\nu=\sum_{i=1}^n d(i)$, be our set of {\em configuration points} and let $\Phi$ be the set of {\em configurations} i.e. functions $\phi:W \mapsto [n]$ that such that $|\f^{-1}(i)|=d(i)$ for every $i \in [n]$. Given $\phi \in \Phi$ we define the graph $G_\phi=([n],E_\phi)$ where $E_\phi=\{\{\phi(2j-1),\phi(2j)\}: j\in [\nu] \}$. Choosing a function $\phi \in \Phi$ uniformly at random yields a random (multi-)graph $G_\phi$ with degree sequence $\bd$. 
\vspace{3mm}
\\\noindent\textbf{Algorithm} \textsc{Reduce}:
\vspace{3mm}
\\ The input is $G_0=G$.\\ 
$i=\hat{\tau}=0$.
\\ \textbf{While} $G_i=(V_i,E_i) \neq (\emptyset,\emptyset)$ do: 
\begin{itemize}
\item[]\textbf{If} $\delta(G_i)=0$: Perform a {\bf vertex-0 removal}: choose a random vertex of degree 0 and remove it from $V_i$.
\item[]\textbf{Else if} $\delta(G_i)=1$: Perform a {\bf vertex-1 removal}: choose a random vertex $v$ of degree 1 and remove it along with its neighbor $w$ and any edge incident to either of $v,w$. 
\item[]\textbf{Else if} $\delta(G_i)=2$: Perform a {\bf contraction}: choose a random vertex $v$ of degree 2. Then replace $\{v\} \cup N(v)$ ($v$ and its neighbors $N(v)$) by a single vertex $v_c$. For $u\in V \setminus (\{v\} \cup N(v))$, $u$ is joined to $v_c$ by as many edges as there are in $G_i$ from $u$ to $\{v\} \cup N(v)$. Here we call $v$ ``the contracted vertex''
\item[]\textbf{Else if } $\delta(G_i)\geq 3$: Perform a {\bf max-edge removal}: choose a random vertex of maximum degree and remove a random edge incident with it.
\\ \textbf{End if}
\item[]\textbf{If} the last action was a max-edge removal, say the removal of edge $\{u,v\}$ and in the current graph we have $d(u)=2$ and $u$ is joined to a single vertex $w$ by a pair of parallel edges then perform an {\bf auto correction contraction}: contract $u,v$ and $w$ into a single vertex.
\\ \textbf{End If}
\item[] Set $i=i+1$ and let $G_i$ be the current graph.
\end{itemize}
\textbf{End While}
\\Set $\hat{\tau}=i.$

Observe that we only reveal edges (pairs of the form $(\phi(2j-1),\phi(2j)): j\in [\nu]$) of $G_\phi$ as the need arises in the algorithm. Moreover the algorithm removes any edges that are revealed. Thus if  we let $\bd(i)$  be the degree sequence of $G_i$ then, given $\bd(i)$  we have that $G_i$ is uniformly distributed among all configurations with degree sequence $\bd(i)$ and no loops.

\subsection{Organizing the actions taken by {REDUCE}}
We do not analyze the effects of each action taken by \textsc{Reduce} individually. Instead we group together sequences of actions, into what we call {\em Hyperactions}, and we analyze the effects of the individual Hyperactions.

The first group of actions consists of all the vertex-0, vertex-1 removals and contractions performed before the first max-edge removal is performed. We let $\Gamma_0$ be the graph resulting from performing the first group of actions. Observe that $\Gamma_0$ is the first graph in the sequence $G_0,G_1,...,G_{\hat{\tau}}$ with minimum degree at least 3. Moreover since every graph that we study in this paper has minimum degree at least 3, in our case we have $G=G_0=\Gamma_0$.

Thereafter, every Hyperaction starts with a max-edge removal and it consists of all the actions taken until the next max-edge removal. We let $\Gamma_{i}$ be the graph that results from performing the first $i$ Hyperactions starting from $\G_0$. Thus $\Gamma_{i}$ is the $(i+1)^{th}$ graph in the sequence $G_0,G_1,...,G_{\hat{\tau}}$ that has minimum degree at least 3 and going from $\Gamma_{i-1}$ to $\Gamma_i$ \textsc{Reduce} performs a max-edge removal followed by a sequence of  vertex-0, vertex-1 removals, contractions and possibly of an auto correction contraction. 
We finally let $\Gamma_{\tau}$ be the final graph. Thus when $G=G_0$ has minimum degree 3 we have that $\Gamma_0,\Gamma_1,...,\Gamma_{\tau}$ is a subsequence of $G_0,G_1,...,G_{\hat{\tau}}$. Furthermore, $\Gamma_0=G_0$,  and $\Gamma_0,\Gamma_1,...,\Gamma_{\tau-1}$ consists of all the graphs in the sequence $G_0,G_1,...,G_{\hat{\tau}}$ with minimum degree at least 3.  
\vspace{3mm}
\\Our analysis mainly focuses on the following Hyperactions: we have put some diagrams of these Hyperactions in Appendix A
\vspace{3mm}
\\\noindent{\bf Hyperactions of Interest:}
\vspace{3mm}
\\\noindent\textbf{ Type 1}: A single max-edge removal,
\\\textbf{ Type 2}: A single max-edge removal followed by an auto correction contraction.
\\\textbf{ Type 3}: A single max-edge removal followed by a good contraction.
\\\textbf{ Type 4}: A single max-edge removal followed by 2 good contractions. 
\vspace{3mm}
\\We divide Hyperactions of Type 3 into three classes. Assume that during a Hyperaction of Type 3 the set $\{v,a,b\}$ is contracted, $v$ is the contracted vertex and $v_c$ is the new vertex. We say that such a Hyperaction is of \textbf{Type 3a} if $d(v_c)=d(a)+d(b)-2$, is of \textbf{Type 3b} if $d(v_c)=d(a)+d(b)-4$ and is of \textbf{Type 3c} if $d(v_c)<d(a)+d(b)-4$. Note that in general, $d(v_c)=d(a)+d(b)-2-2\eta_{a,b}$, where $\eta_{a,b}$ is the number of edges joining $a,b$.
\vspace{3mm}
\\With the exception of a Hyperaction of Type 3c, where $\eta_{a,b}\geq 2$, we refer to the Hyperactions of interest as good Hyperactions. We call any Hyperaction that is not good, including a Hyperaction of Type 3c, bad. 
\subsection{The excess}\label{list} 
For a graph $G$ and a positive  integer $\ell$ we let 
$$ ex_\ell(G):=\sum_{v \in V(G)} [d(v)-\ell]\mathbb{I}(d(v)>\ell).$$
Furthermore for $i\leq \tau$ we let $ex_{\ell,i}=ex_{\ell}(\G_i)$.
We use $ex_{\ell}$ to control the Hyperactions taken by Reduce.
Lemma \ref{hyper} implies that as long as $ex_k$ stays small, \textsc{Reduce} performs only good Hyperactions. Later on, at Lemma \ref{general} and Corollary \ref{stopF} we argue that as long as only good Hyperactions are performed $ex_k$ stays small.

\begin{lem}\label{lem31}(Lemma 3 of \cite{AF}) 
Let $i\geq 0$ and  assume  that $\G_i$ satisfies 
$ex_{\ell}(\G_i) \leq \log |V(\G_i)|$ for  some $3\leq \ell =O(1)$. 
Then  with  probability $1-o(|V(\G_i)|^{-1.9})$ the  Hyperaction  that \textsc{reduce} applies  to $\G_i$ is good.  In  addition,  a  Hyperaction  of  Type  2,3b or 4 is applied with  probability
$o(|V(\G_i)|^{- 0.9})$.
\end{lem}
Given Lemma \ref{lem31} we can now prove the following:
\begin{lem}\label{hyper}
For $\ell \in \mathbb{N}$ let $Q_\ell(G)$ be  the event that  \textsc{Reduce} applies  good Hyperactions  to every 
graph $\G'$ of the sequence  $\G_0,\G_1,...,\G_{\tau}$ that satisfies $ex_{\ell}(\G') \leq  \log^2 n$,  $\Delta(\G')>3$ and $|E(\G')|\geq n^{0.9} $. Then $$\pr(Q_\ell(G))=1-o(n^{-0.5}).$$  
Furthermore if $\G'$ is such a graph then \textsc{Reduce} applies to $\G'$ a bad Hyperaction   with probability $o(n^{-1.75})$ while it  
applies to $\G'$ a Hyperaction 
of Type 2, 3b or 4  with probability $o(n^{-0.75})$. 
\end{lem}
\begin{proof}
$ex_{\ell}(\G') \leq  \log^2 n$ implies that $|E(\G')| \leq \ell |V(\G')|/2+\log^2 n$. Thus for $r \geq {0.9}, |E(\G_i)| \geq n^{0.9}$
\begin{align}\label{eq4local}
|V(\G_i)|^{r} \geq  (|2E(\G_i)|-2\log^2 n)/\ell)^{r} \geq (|E(\G_i)|/\ell)^{r}.
\end{align}
 $|E(\G_i)|$ is decreasing with respect to $i$. Therefore the probability the event $Q_\ell(G)$ does not occur is bounded by
\begin{align*}
\sum_{i:|E(\G_i)|\geq n^{0.9}} |V(\G_i)|^{-1.9} 
&\geq \sum_{i:|E(\G_i)|\geq n^{1.9}} (|E(\G_i)|/\ell)^{-1.9}
\geq \sum_{ j \geq n^{0.9}} \ell^{-1.9} j^{-1.9} =o(n^{-0.5}). 
\end{align*}
The second part of Lemma \ref{hyper} follows directly from Lemma \ref{lem31}, the inequality $e_i\geq n^{0.9}$ and \eqref{eq4local}.
\end{proof}
\subsection{Proof of Theorem \ref{main}}
The proof of Theorem \ref{main} follows from Lemmas \ref{ind}, \ref{567} and \ref{34} given below. 
For $\ell\in \mathbb{N}$ define the stopping times  
$$\tau_\ell:=\min\{i: \G_i \text{ has maximum degree } \ell \text{ or }
|E(\G_i)| \leq n^{0.9}\}.$$
\begin{lem}\label{ind}
Let $8\leq k =O(1)$ and $\G_0=G\in \cC_{3,k}$ be a random (multi)-graph with degree sequence $\bd$, ,  maximum degree k, minimum degree 3 and no loops. Then w.h.p.\@
%ith probability $1-o(n^{-0.5})$, 
\begin{itemize}
\item[]i) the first $\tau_{k-1}-1$ Hyperactions applied by \textsc{Reduce} on $\G_0$ are good,
\vspace{-2mm}
\item[]ii) $\G_i\in \cC_{3, k-1}$ for $i\leq \tau_{k-1}$,
\vspace{-2mm}
\item[]iii) $|E(\G_{\tau_{k-1}})| \geq (1-4/k)|E(\G_0)|=\Omega(n)$.
\end{itemize}
\end{lem}
\noindent Critical elements of the proof of Lemma \ref{ind} are Lemmas \ref{remA} and Lemma \ref{halfb}. The proof of Lemma \ref{propb} is given in section \ref{sind} while the proof of Lemma \ref{remA} is given in Appendix B. 
\begin{lem}\label{remA} Let $p_{3,i}=3n_3/2e_i$ and assume that $e_i\geq n^{0.9}$. 
Then $\G_i \in \cC_{3,k-1}$ implies that 
$$p_{3,i} \leq  \frac{3}{\sum_{j=3}^{k-1} \ca^{j-3}j}+o(1).$$
Hence $p_{3,i}\leq 0.081$  for $k\geq 8$.
\end{lem}
\begin{lem}\label{propb}
Let $8\leq k=O(1)$. 
Let $\G_i$ be  the  first  graph  in  the  sequence
$\G_0,\G_1,...,\G_{\tau}$ that  does  not  belong  to $\cC_{3,k-1} \supseteq \cC_{3,k}$.
Then  w.h.p. $\G_i$ violates  only  $D_{k-1,k-1}(G_i)$  out of the events $D_{k-1,k-1}(G_i),D_{k-1,k-2}(G_i)...,D_{k-1,4}(G_i)$  and it belongs  to
$\cC_{3,k -2} \supseteq \cC_{3,k-1}$.
\end{lem}
\noindent For $\ell \in \mathbb{N}$ let
$$t_\ell:=\min\{i: \G_i \notin \cC_{3,\ell} \text{ or } |E(\G_i)|< n^{0.9}\}.$$
We will show that most of the time only Hyperactions of Type 1 or 3a occur. A Hyperaction of Type 1 decrease $ex_{k,i}$ by at least 1 while a Hyperaction of Type 3a increase $ex_{k,i}$ by at most $k-2$ (see \eqref{bddex}, Lemma \ref{multbounds}). The later occus with probability at most $p_{3,i}$. Hence we can use Lemma \ref{remA} to bound $p_{3,i}$ and show that if $\G_i \in \cC_{k-1,3}$ and $ex_{k,i}>0$ then $ex_{k,i}$ will decrease in expectation. It will follow (done in Lemma \ref{general} and Corollary \ref{stopF}) that $ex_{k,i} \leq \log^2 n$ for $i<t_{k-1}$.
 Thus Lemma \ref{hyper} implies that for $i < t_{k-1}$ the $i^{th}$ Hyperaction is good.
   
Thereafter using Lemma \ref{propb} we show that $\tau_{k-1} \leq t_{k-1}$.  Lemma \ref{propb} implies that w.h.p.\@ the hitting time of $\neg \cD_{k-1,k-1}(\G_i)$ and $t_{k-1}$ are equal. The first one has simpler description thus it is easier to monitor.
\vspace{3mm}
\\Unfortunately not all the calculation done in the proof of Lemma \ref{ind} are valid for smaller values of $k$. The following Lemma aims to treat those cases.
\begin{lem}\label{567}
Let  $k\in\{5,6,7\}$ and $\G_0=G$ be a random (multi)-graph with degree sequence $\bd$, maximum degree $k$, minimum degree 3 and no loops.  Then w.h.p.\@
%ith probability $1-o(n^{-0.5})$  
\begin{itemize}
\item[]i) the first $\tau_{k-1}-1$ Hyperactions applied by \textsc{Reduce} on $\G_0$ are good,
\vspace{-2mm}
\item[]ii) $|E(\G_{\tau_{k-1}})|\geq |E(\G_0)|/10^{25}=\Omega(n)$.
\end{itemize}
\end{lem}
\begin{lem}\label{34}
Let  $k=O(1)$ and $\G_0=G$ be a random (multi)-graph with degree sequence $\bd$, maximum degree $k$, minimum degree 3, no loops and even number of vertices.  If w.h.p.\@ there exist $i$ such that
%ith probability $1-o(n^{-0.5})$  
\begin{itemize}
\item[]i) the first $i$ Hyperactions performed by \textsc{Reduce} are good,
\vspace{-2mm}
\item[]ii) $\G_i$ has minimum degree 3 and maximum degree 4 
\vspace{-2mm}
\item[]iii) $|V(G_i)| = \Omega(n)$,
\end{itemize}
then, w.h.p.\@
% with probability $1-o(n^{0.5})$
 \textsc{Reduce-Construct} finds a perfect  matching of $G$ in $O(n)$ time. 
\end{lem}
\begin{proof}
Indeed if (i) occurs then the first $i$ Hypractions are good. Therefore no bad contractions  or 0-vertex removals are performed.
Hence Lemma 7 of \cite{AF} implies that \textsc{backtrack} extends any maximum matching of $\G_i$ to a maximum matching of $\G_0$. 
Finally conditions (ii), (iii) and Theorem 1 of \cite{AF} imply 
that  w.h.p.\@
% with probability $1-o(n^{0.5})$
\textsc{Reduce-Construct} finds a perfect  matching of $\G_i$ in $O(n)$ time. 
\end{proof}
\noindent{\textbf{ Proof of Theorem 1:}}
The case $k=3,4$ follows directly from Lemma \ref{34} with $i=0$, while for $k=5,6,7$ 
by iteratively applying Lemma \ref{567} we can find $i$ that satisfies the conditions
of Lemma \ref{34}. 

Given Lemma \ref{567} for $k\geq 8$ it suffices to prove the following statement:
 Let $\G_0=G\in \cC_{3,k}$ be a random graph with degree sequence $\bd$, maximum degree $k$, minimum degree 3 and no loops. Then w.h.p.\@
 % with probability $1-o(n^{-0.5})$ 
\begin{itemize}
\item[] i) the first $\tau_{7}$ Hyperactions are good,  
\vspace{-2mm}
\item[] ii)  $|E(G_{\tau_k-1})|=\Omega(n).$ 
\end{itemize}
We proceed to prove the above statement by induction. The base case follows from Lemma \ref{ind} with $k=8$. Its inductive step also follows directly from Lemma \ref{ind} and the inductive Hypothesis.   
\qed
\subsection{Overview}
In Subsection \ref{pr} we study how the good Hyperactions effect the expected changes of $n_{r,i}$ and $ex_{\ell,i}$ respectively (done in Lemmas  Lemmas \ref{multbounds1} and \ref{multbounds}).
Later using the expected change of $ex_{\ell,i}$ we show that for $i<\hat{t}_k$ the $i^{th}$ Hyperaction is good. Here, $\hat{t}_k=t_{k-1}$ if $k\geq 8$ and $\hat{t}_k=\tau_{k-1}$ if $k=5,6,7$.

We prove Lemmas \ref{ind} and \ref{567} in Sections \ref{sind} and \ref{s567} respectively.
For the proof of Lemma \ref{ind} we start by showing that at time $t_{k-1}$ the event $\cD_{k-1,k-1,t_{k-1}}$ occur (done in Lemma \ref{halfb}). We let $X_{r,i}=n_{r,i}-\ca n_{r-i,i}$ for $4\leq r\leq k-2$. Thus $X_{r,i}>-[\log^2 n -(k-1)]n^{0.8}/2^r$ for $i<t_{k-1}$. We argue that if $i<t_{k-1}$ and $X_{r,i}$ is close to $[\log^2 n -(k-1)]n^{0.8}/2^r$ then after the $i^{th}$ Hyperaction it will increase in expectation. In Lemma \ref{halfb2} we use a similar argument 
in order to show that $X_{k-1,i}=n_{k-1,i}-\ca n_{k-1-i,i}>-[\log^2 n -(k-1)]n^{0.8}/2^{k-1}$ for $i\leq \tau_{k-1}$. 

For the proof of Lemma \ref{567} we compare the rate of decrease of $ex_{k,i}$ and $e_i$. We show that the first one is larger and we argue that it will reach 0 (done at time $\tau_{k-1}$) before the number of edges becomes sublinear.
\section{Notation - Preliminaries Results}
In the first subsection of this section we  collect  various pieces of notation that we will use for ease  of  reference. In the second subsection we prove results that either concern the Hyperactions or are used in multiple Sections.
\subsection{Notation}
 $\Gamma_0,\Gamma_1,...,\Gamma_{\tau}$ is a sequence of graphs that is generated by \textsc{Reduced}. $\Gamma_0=G$ is the input and $\Gamma_{\tau}$ is the empty graph. Every graph in the sequence has minimum degree 3 except the last one. To go 
from $\G_i$ to $\G_{i+1}$ \textsc{Reduce} performs a Hyperaction which may be one of the Hyperactions of Interest (a.k.a.  good Hyperactions), listed in Subsection 2.2. 
Furthermore as pointed in \cite{AF} given the degree sequence $\bd_i$ of $\G_i$ we have that $\G_i$ is a random (multi)-graph with degree sequence $\G_i$ and no loops.
\vspace{3mm}
\\ Observe that at every Hyperaction  a max-edge removal is performed, therefore $2e_i \leq 2e_0-2i$ for $i \leq \tau$.
Thereafter if our initial graph has $n$ vertices and  maximum degree $k$ then 
$2e_i\leq kn -2i$. Since $e_\tau=0$  every $i$ smaller than $\tau$ satisfies
\begin{align}\label{ed1}
 i\leq \tau \leq kn/2.
\end{align}
For a graph $G$, $j,\ell \in \mathbb{N}$ we let:
\begin{itemize}
\item $\delta(G)$ and $\Delta(G)$ be the minimum and maximum degrees of $G$ respectively, 
\item $n_j(G)$ is the number of vertices of $G$ of degree $j$,
\item $p_j(G):=\frac{jn_j(G)}{2|E(G)|}$,
\item $p_{>j}(G):=\sum_{h>j} p_h(G)$,
\item $ ex_\ell(G):=\sum_{v \in V(G)} [d(v)-\ell]\mathbb{I}(d(v)>\ell)$, 
\end{itemize}
We denote by $\delta_i,\Delta_i,n_{j,i},p_{j,i},p_{>j,i}$ and $ex_{\ell,i}$ the corresponding quantities of $\G_i$.  Furthermore we let $e_j:=|E(\G_j)|$.
\vspace{3mm}
\\We let $\ca=1.17$. For $\ell \in \mathbb{N}$ we define the set  
of $(\ca,3,\ell)$-\emph{dominant} random graphs $\cC_{3,\ell}$ by
$$\cC_{3,\ell}:=\{G: \cD_{\ell,j}(G) \text{ holds  for all } 3<j\leq \ell\}$$
where 
$$\cD_{\ell,j}(G):=\{ n_j(G) \geq \ca n_{j-1}(G) - (\log^2 n -\ell)n^{0.8}/2^j\}.$$
Observe that from the above definition follows that $ \cC_{3,\ell} \subseteq \cC_{3,\ell-1}$. We denote $\cD_{\ell,j}(\G_i)$ by $\cD_{\ell,j,i}$.
\vspace{3mm}
\\ 
Given the sequence $\Gamma_0,\Gamma_1,...,\Gamma_{\tau}$, for $3\leq j =O(1)$ we define the following stopping times
\begin{itemize}
\item $\tau_j:=\min\{i:  \G_i\text{ has maximum degree }j \text{ or }
e_i< n^{0.9}\}$, 
\item $t_j:=\min\{i: \G_i \notin \cC_{3,j} \text{ or }
e_i < n^{0.9}\}$ and
\item $t^*_{j}=\min\{\tau_j,t_j\}$.
\end{itemize}
We later show that if $\G_0\in \cC_{3,k} \subseteq \cC_{3,k-1}$ then
 w.h.p.   $\tau_{k-1}=\tau_{k-1} < t_{k-1}$.
\vspace{3mm}
\\For $\ell,j \in \mathbb{N}$ we let $F_{\ell,j}(G)$  be the event that
\begin{itemize}
    \item[] i) $ex_{\ell,i} \leq  \log^2 n$ for $0\leq i < j$,
    \vspace{-2mm}
    \item[] ii) \textsc{Reduce} applies a good  Hyperaction to $\G_i$ for $0\leq i < j$
    \vspace{-2mm}
    \item[] iii) for every $i\leq j$ there exists $z_i$, $i-\log^2 n/ (k-2) \leq z_i < i$ such that $ex_{\ell,z_i}=0$
\end{itemize}
\noindent We let 
 $Q_\ell(G)$ be  the event that  \textsc{Reduce} applies
a good Hyperaction to every 
graph $\G'$ of the sequence  $\G_0,\G_1,...,\G_{\tau}$ that satisfies 
$ex_{\ell}(\G') \leq \log^2 n$,  $\Delta(\G')>3$ and $e(\G')\geq n^{0.9} $.
\vspace{3mm}
\\In multiple places we are going to use  the Azuma-Hoeffding inequality (see \cite{AZ}, \cite{H}). \begin{lem}\label{ah}
Let $b \in \mathbb{N}$. 
For $i < \tau_3$ let $X_i$ be a random variable that  is bounded by $b$.
Let $Y_i'=  \E{X_i|\cF_i},$
 where $\cF_i$ is the filtration determined by $\G_0=G$ and the first $i$ Hyperactions.
Let $Y_i=  \E{X_i|\G_i},$ and assume that  $Y_i'=Y_i$ for all $i\leq \tau_3$.
Then for any $t>0$ and any $0\leq i_1<i_2 < \tau_3$
$$\pr\bigg(\bigg|\sum_{j=i_1}^{i_2-1}( Y_j- X_j) \bigg|>t\bigg) \leq 2\exp \bigg\{-\frac{t^2}{2 b^2 (i_2-i_1) }\bigg\}.$$
\end{lem}
\subsection{Preliminary Results}\label{pr}
In Lemmas \ref{multbounds1} and \ref{multbounds} we study how the good Hyperactions effect the expected changes of $n_{r,i}$ and $ex_{\ell,i}$ respectively. 
As discussed earlier given the degree sequence of $G_i$, $\bd(i)$,  we have that $G_i$ is uniformly distributed among all configurations with degree sequence $\bd(i)$ and no loops.
The mild conditioning resulting from imposing the condition that $\G_i$ has no loops 
is insignificant and results in $(1+o(1))$ 
factors. For clarity of the presentation we omit such factors.
%Before diving into the calculations observe that if we apply a good Hyperaction to $\G_i$ then only a constant number of vertices {\red{ the degree sequence changes going from $\G_i$ to $\G_{i+1}$ when we calculate various probabilities  we should account for i) the mild conditioning resulted by imposing the condition that $\G_i$ has no loops  and ii)  the fact that we may have revealed some edges of $\G_i$. The effect of  i) and ii) on expressions like $p_{i,j}$ and $1-p_{i,j}$ will be insignificant resulting into $(1+o(1))$ factors that do not affect the calculations in any meaningful way.  For clarity of the presentation we omit such $(1+o(1))$ factors.}}
\begin{lem}\label{multbounds1}
Let $4< k= O(1)$. 
Let $\Gamma_0=G$ be a random (multi)-graph with degree sequence $\bd$, maximum degree $k$, minimum degree 3 and no loops. Suppose that 
$\Gamma_i$ satisfies $ex_{k,i}\leq   \log^2 n$,
has maximum degree at least $k$ and $e_i \geq n^{0.9}$. 
Conditioned on the event that a good Hyperaction is applied to $\G_i$, 
then for $3\leq r \leq k-1$ 
\begin{align}\label{bddv}
|n_{r,i+1}-n_{r,i}|\leq 5.
\end{align}
Furthermore for $3 \leq r \leq k-2,$ 
\begin{align}\label{ev}
\E{n_{r,i+1}-n_{r,i}|\G_i} &=
 p_{r+1,i}-p_{r,i} 
 +p_{3,i}\bigg[\sum_{j_1+j_2-2=r} p_{j_1,i}p_{j_2,i}       
-2p_{r,i} \bigg] 
 +o(n^{-0.75}),
\end{align}
and for $r=k-1$,
\begin{align}\label{ev2}
\E{n_{k-1,i+1}-n_{k-1,i}|\G_i} &=
 p_{k,i}-p_{k-1,i} 
 +p_{3,i}\bigg[\sum_{j_1+j_2-2=k-1} p_{j_1,i}p_{j_2,i}       
-2p_{k-1,i} \bigg]
\\& +\mathbb{I}(\Delta_i=k) 
 +o(n^{-0.75}).\nonumber 
\end{align}
In addition 
\begin{align}\label{bdedges}
|e_{i+1}-e_{i}|\leq 6,
\end{align}
and
\begin{align}\label{edgedrift}
\E{e_{i+1}-e_{i}|\G_i} = -1-2p_{3,i} + o(n^{-0.75})
\end{align}
\end{lem}
\begin{proof}
Fix $r$, $3\leq r \leq k-2$.
Throughout this Lemma we condition on the event that the $i^{th}$ Hyperaction is good. 
$\Delta_i\geq k > 3$ thus \textsc{Reduce} proceeds and performs a Hyperaction of Type 1,2,3a,3b or 4 with probability  $1-p_{3,i}, o(n^{-0.75}), p_{3,i}, o(n^{-0.75})$ and $o(n^{-0.75})$ respectively. All of the Hyperactions start with a max-edge removal. That is 
a random vertex of maximum degree $v$ is chosen along with a random neighbor of it $u$ and the edge $\{v,u\}$ is removed. $v$ is a vertex of maximum degree and thus $d(v)=\Delta_i \geq k$. We summarize the case analysis that follows at Tables 1 and 2 given below.
\vspace{3mm}
\\If $d(u)>3$ then a Hyperaction of Type 1  occurs. If
$d(u) \neq d(v)$ then as a result $\G_i$ ``loses" a vertex of degree $d(v)$, a vertex of degree $d(u)$ and then, it gains  
a vertex of degree $d(v)-1$ and a vertex of degree $d(u)-1$. If $d(u)=d(v)$
then   $\G_i$ ``loses" 2 vertices of degree $d(v)$  and   it gains  
two vertex of degree $d(v)-1$. Given the above for the  change  of
$n_{r,i}$ we have the following two cases:

$\bullet$ \textbf{Case a:}$d(u)=r+1$. Then, $n_{r,i+1}-n_{r,i}=1$. Case (a) occurs with probability $p_{r+1,i}$.

$\bullet$ \textbf{Case b:} $d(u)=r$. Then, $n_{r,i+1}-n_{r,i}=-1$. Case (b) occurs with probability $p_{r,i}$.

If $d(u)=3$ then a Hyperaction of Type 2,3a,3b or 4 occurs. Assume that a Hyperaction of Type 3a occurs, that is  $u$ has 3 neighbors in $\G_i$, let them be $\{v,x_1,x_2\}$ 
and there is no edge from $x_1$ to $x_2$.  In this case 
\textsc{Reduce} proceeds and contracts $\{v,x_1,x_2\}$. The new vertex, say $v_c$,  has degree $d(x_1)+d(x_2)-2$.
For $j\geq 3$   the number of vertices of degree $j$, $n_{j,i}$ is decreased by 1 for every vertex of degree $j$ in $\{v,u,x_1,x_2\}$. Then, for $j\geq 3$, $n_{j,i}$ is increased by 1 for every element of $\{d(v)-1,d(x_1)+d(x_2)-2\}$ that is equal to j.
First we let $4\leq r \leq k-1$  and we consider the following 3 cases: 

$\bullet$ \textbf{Case c:} $d(x_1)=d(x_2)=r$. Then  $n_{r,i+1}-n_{r,i}=-2$.
Case (c) occurs with probability $p_{3,i}p_{r,i}^2$. 

$\bullet$ \textbf{Case d:} $d(x_1)=r,d(x_2)\neq r$ or $d(x_1)\neq r,d(x_2)= r$ . Then  $n_{r,i+1}-n_{r,i}=-1$.
Case (d) occurs with probability $p_{3,i}2p_{r,i}(1-p_{r,i})$. 

$\bullet$ \textbf{Case e:} $d(x_1)+d(x_2)-2=r$. Then the new vertex has degree $r$ and  $n_{r,i+1}-n_{r,i}=1$.
Case (e) occurs with probability $p_{3,i}\cdot p_{d(x_1),i} \cdot p_{d(x_2),i}$.

If $r=3$ then the above cases are modified as follows (recall that $d(u)=3$): 

$\bullet$ \textbf{Case c':} $d(x_1)=d(x_2)=3$. Then
 $n_{3,i+1}-n_{3,i}=-3$.
Case (c') occurs with probability $p_{3,i}^3$.
 
$\bullet$ \textbf{Case d':} $d(x_1)=3,d(x_2)\neq 3$ or $d(x_1)\neq 3,d(x_2)= 3$.
Then  $n_{3,i+1}-n_{3,i}=-2$.
Case (d') occurs with probability $2p_{3,i}^2(1-p_{3,i})$.

$\bullet$ \textbf{Case e':} $d(x_1),d(x_2)>3$.
Then  $n_{3,i+1}-n_{3,i}=-1$.
Case (e') occurs with probability $p_{3,i}(1-p_{3,i})^2$.

From the case analysis above and the definition of the Hyperactions  it follows that \eqref{bddv} holds for $3\leq r\leq k-2$. The upper bound  in \eqref{bddv} is achieved when a Hyperaction of Type 4 takes place and all 5 vertices involved in the contractions have degree 3. 

For $4\leq r\leq k-2$ we summarize the case analysis in Table 1.  
\begin{table}[h]
\centering
    \begin{tabular}{| c | c | c | c | c |}
    \hline
  \textbf{Case} &${d(u)}$ & \textbf{Hyperaction  that} & $n_{r,i+1}-n_{r,i}$ & \textbf{probability occurring}
 \\ & &   \textbf{takes place}& &
     \\ \hline
 Case a & $r+1$ & Type 1 &1 & $p_{r+1,i}$ \\ \hline
 Case b & $r$ & Type 1 &-1& $p_{r,i}$ \\ \hline
 Case c & 3 & Type 3a &-2& $p_{3,i}p_{r,i}^2$ \\ \hline
 Case d & 3 & Type 3a &-1& $2p_{3,i}^2(1-p_{r,i})$ \\ \hline
 Case e & 3 & Type 3a &1& $p_{3,i}\sum_{j_1+j_2-2=r}p_{j_1,i}p_{j_2,i}$ \\ \hline
    \end{tabular}
    \caption{Case analysis for $4\leq r \leq k-2$}
\end{table}
\\A Hyperaction of Type 2,3b or 4 occurs with probability $o(n^{-0.75})$.  Thus,
\begin{align*}
\E{n_{r,i+1}-n_{r,i}|\G_i}& =
 p_{r+1,i}-p_{r,i} -2p_{3,i}p_{r,i}^2 -2p_{3,i}p_{r,i}(1-p_{r,i})
 \\& 
 +p_{3,i}\sum_{j_1+j_2-2=r} p_{j_1,i}p_{j_2,i}        
 +o(n^{-0.75})
\\&=
 p_{r+1,i}-p_{r,i} 
 +p_{3,i}\bigg[\sum_{j_1+j_2-2=r} p_{j_1,i}p_{j_2,i}       
-2p_{r,i} \bigg] 
 +o(n^{-0.75}).
\end{align*}
If $r=3$ then Case (b) does not apply since in Case b we assume that $d(u)>3$. In place of Table 1 we have Table 2 given below.
\begin{table}[h]
\centering
    \begin{tabular}{| c | c | c | c | c |}
    \hline
  \textbf{Case} &${d(u)}$ & \textbf{Hyperaction  that} & $n_{3,i+1}-n_{3,i}$ & \textbf{probability occurring}
 \\ & &   \textbf{takes place}& &
     \\ \hline
 Case a & 4 & Type 1 &1 & $p_{4,i}$ \\ \hline
 Case c' & 3 & Type 3a &-3& $p_{3,i}^3$ \\ \hline
 Case d' & 3 & Type 3a &-2& $2p_{3,i}^2(1-p_{3,i})$ \\ \hline
 Case e' & 3 & Type 3a &-1& $p_{3,i}(1-p_{3,i})^2$ \\ \hline
    \end{tabular}
    \caption{Case analysis for $r=3$}
\end{table}
\\A Hyperaction of Type 2,3b or 4 occurs with probability $o(n^{-0.75})$.  Thus,
using the identity   $p_{3,i}\sum_{j_1+j_2-2=3} p_{j_1,i}p_{j_2,i} =0$ ($p_{j,i}=0$ for $j<3$)
 we have,
\begin{align*}
\E{n_{3,i+1}-n_{3,i}|\G_i}& =
 p_{4,i}-3p_{3,i}^3 -4p_{3,i}^2(1-p_{r,i})-p_{3,i}(1-p_{3,i})^2
\\& +p_{3,i}\sum_{j_1+j_2-2=3} p_{j_1,i}p_{j_2,i}        
 +o(n^{-0.75})
\\&=  p_{4,i}-p_{3,i}  +p_{3,i}\bigg[\sum_{j_1+j_2-2=3} p_{j_1,i}p_{j_2,i}       
-2p_{3,i} \bigg] 
 +o(n^{-0.75}).
\end{align*}
For the derivation of \eqref{bddv} for $r=k-1$ and \eqref{ev2} the same analysis applies  modulo the fact 
that  if $d(v)=\Delta_i=k$, i.e. $\mathbb{I}(\Delta_i=k)=1$,
 then due to the max-edge removal  $n_{k-1,i}$ is initially increased by one resulting to the additional $\mathbb{I}(\Delta_i=k)$ term found in \eqref{ev2}.
 
\eqref{bdedges} follows from the definition of the Hyperactions. 

Finally for \eqref{edgedrift} we have the following table:
\begin{table}[h]
\centering
    \begin{tabular}{| c | c | c |}
    \hline
  \textbf{Hyperaction} & $e_{i+1}-e_{i}$ & \textbf{probability occurring}
     \\ \hline
Type 1 & -1&  $1-p_{3,i}$ \\ \hline
Type 3a & -3 &  $p_{3,i}$ \\ \hline
Type 2/3b/4 & O(1) & $o(n^{-0.75})$ \\ \hline
    \end{tabular}
\end{table}
\\Therefore,
 $$\E{e_{i+1}-e_i}= -(1-p_{3,i})-3 p_{3,i}- o(n^{-0.75})= -1-2p_{3,i}+o(n^{0.75}).$$
\end{proof}
\begin{cor}\label{new1}
Let $i_1<i_2$. Assume that the first $i_2-1$ Hyperactions are good. Then $i_2-i_1\leq n^{0.8}$ and $e_{i_2-1}\geq n^{0.9}$ imply that $|p_{j,i_2}-p_{j,i_1}|\leq o(n^{-0.05})$  for all $j \in \mathbb{N}$.
\end{cor}
\begin{proof}
Follows directly from \eqref{bdedges}.
\end{proof}
In the proof of Lemma \ref{ind} we will invoke \eqref{ev2} to control $n_{k-1,i_1}-n_{k-1,i_2}$ for some $i_1,i_2$. We use the following lemma to control the change of  the most problematic term appearing in \eqref{ev2}, namely of the term  $\mathbb{I}(\Delta_i=k)$

\begin{lem}\label{new2}
Let $8\leq k=O(1)$. Assume $i_1,i_2$ satisfy  $i_1 < i_2$, $n^{0.7} \leq i_2-i_1 \leq n^{0.8},$ $e_{i_2-1}\geq n^{0.9}$ and the event $F_{k-1,i_2}$ occurs. Then w.h.p. 
\begin{align}
\sum_{i=i_1}^{i_2-1} \mathbb{I}(\Delta_i=k) \geq (i_2-i_1)[0.999 - (k-2)p_{3,i_1}] .
\end{align}
\end{lem}
\begin{proof}
In the event $F_{k,i_2}$ occurs  for $i_1 \leq i< i_2$ the inequality $ ex_{k,i}< \log^2 n$  holds and the $i^{th}$ Hyperaction is good. 

Let $Z_i=\ell-k$ if the $i^{th}$ Hyperaction  is of Type 2,3a,3b or 4  and the new vertex has degree $\ell>k$.
Otherwise let $Z_i=0$.

$ex_{k,i}<\log^2 n$ implies that $p_{\ell,i} \leq (\log^2n +k)/2e_i =o(n^{-0.75})$ for $\ell \leq \log^2n +k$ and $p_{\ell,i} =0$ for $\ell>\log^2 n+k$.
 Hence for $\ell>k,$ 
$$\pr(Z_i=\ell-k)= p_{3,i} \sum_{d_1+d_2-2=\ell} p_{d_1,i}p_{d_2,i}+o(1)= p_{3,i} \sum_{ \substack{d_1+d_2-2=\ell\\ 3\leq d_1,d_2 \leq k}} p_{d_1,i}p_{d_2,i}+o(1).$$
The $o(1)$ term accounts for the event that the $i^{th}$  Hyperaction  is of Type 2,3b or 4.
Thus,
\begin{align}
\E{Z_i|\G_i} &=    p_{3,i}\sum_{3\leq j_1,j_2 \leq k} p_{j_1,i}p_{j_2,i} \cdot(j_1+j_2-2-k)\mathbb{I}(j_1+j_2-2-k\geq 0)+o(1) \label{Z1}
\\& \leq (k-2)  p_{3,i} +o(1) \leq (k-2)  p_{3,i_1} +o(1). \label{Z}
\end{align}
At \eqref{Z} we used corollary \ref{new1}.
Now observe that 
\begin{align}\label{z2}
0\leq ex_{k,i_2} = ex_{k,i_1} + \sum_{i=i_1}^{i_2-1} (ex_{k,i+1}-ex_{k,i})\leq  ex_{k,i_1} + \sum_{i=i_1}^{i_2-1}\{ Z_i  -[1-\mathbb{I}(\Delta_i=k)]\}.
\end{align}
The $-[1-\mathbb{I}(\Delta_i=k)]$ accounts for the fact that when $\Delta_i>k$, due to the max-edge removal, $ex_{k,i}$ is decreased by 1.
In the event $\sum_{i=i_1}^{i_2-1} \mathbb{I}(\Delta(\G_i)=k) < (i_2-i_1)[0.999 - (k-2)p_{3,i_1}]$,
\eqref{z2} implies
\begin{align*}
(i_2-i_1)[0.999 - (k-2)p_{3,i_1}]&> \sum_{i=i_1}^{i_2-1} \mathbb{I}(\Delta_i=k)
\geq (i_2-i_1)-ex_{k,i_1}- \sum_{i=i_1}^{i_2-1}Z_i 
\\& \geq (i_2-i_1)-\log^2 n- \sum_{i=i_1}^{i_2-1}Z_i.
\end{align*}
Hence
\begin{align} \label{Z3}
\sum_{i=i_1}^{i_2-1}Z_i \geq (i_2-i_1)[0.001+(k-2)p_{3,i_1}]-\log^2 n.
\end{align}
Since $|Z_i| \leq k-2$ and\eqref{Z} and  \eqref{Z3} hold the Azuma-Hoeffding inequality implies
\begin{align*}
\pr\bigg(&\sum_{i=i_1}^{i_2-1} \mathbb{I}(\Delta(\G_i)=k) < (i_2-i_1)[0.999 - (k-2)p_{3,i_1} \bigg) 
\\ &\leq   2\exp\bigg\{\frac{[(0.001-o(1))(i_2-i_1)]^2}{2(k-2)^2(i_2-i_1)}  \bigg\}=o(n^{-0.5}).
\end{align*}
In the last equation we used that $n^{0.7}\leq i_2-i_1$.
\end{proof}
We now proceed to calculate the expected change of $ex_{\ell,i}$ in terms of $p_{3,i}$ and $p_{\ell+1}$. Later at Lemma \ref{general} we argue that as long as it is negative only good Hyperactions occur.
\begin{lem}\label{multbounds}
Let $4< k= O(1)$. 
Let $\Gamma_0=G$ be a random (multi)-graph with degree sequence $\bd$, maximum degree $k$, minimum degree 3 and no loops. Suppose that 
$\Gamma_i$ satisfies $ex_{k,i}\leq   \log^2 n$,
has maximum degree at least $k$ and $e_i \geq n^{0.9}$. 
Conditioned on the event that a good Hyperaction is applied to $\G_i$ for $ 3 \leq \ell \leq k$
\begin{align}\label{bddex}
|ex_{\ell,i+1}-ex_{\ell,i}|\leq \ell - 3+ \mathbb{I}(ex_{\ell,i}=0).
\end{align}
Moreover, either $ex_{\ell,i}=0$ or 
\begin{align}\label{elldrift}
\E{ex_{\ell,i+1}-ex_{\ell,i}| \G_i}\leq  -(1-p_{3,i})-p_{\ell+1,i} -p_{3,i}^3 +(\ell-3) p_{3,i} (1-p_{3,i})^2+n^{-0.75}.
\end{align}
\end{lem}
\begin{proof}
As in Lemma \ref{multbounds1} we condition on the event that the $i^{th}$ Hyperaction is good. 
The case analysis is summarized at Table 3 given below. Fix $3\leq \ell \leq k$

If $d(u)>3$ then a Hyperaction of Type 1 occurs. A vertex of maximum degree $v$ is chosen along with a neighbor of it $u$. Then $\{v,u\}$ is removed. $ex_{\ell,i}$ is decreased by 1 for every vertex of degree larger than $\ell\geq 3$ in $\{v,u\}$.
$d(v)=\Delta_i$ and therefore $d(v)=\Delta_i>\ell$ if and only if $ex_{\ell,i}>0$. Based on $d(u)$ we consider the following cases:

$\bullet$ \textbf{Case 1:} $\ell+1\leq d(u)$. $d(v)\geq d(u)\geq \ell$ implies that if Case 1 occurs then $ex_{\ell,i+1}-ex_{\ell,i}=-2$.

Case 1 occurs only if $\ell <d(u)$ and hence with probability $\pr(d(u)\geq \ell+1)=p_{\geq \ell+1,i}$.

$\bullet$ \textbf{Case 2:} $3<d(u) \leq \ell$. If Case 2 occurs then $ex_{\ell,i+1}-ex_{\ell,i}=-\mathbb{I}(ex_{\ell,i}>0)$. Case 2 occurs only if $3<d(u) \leq \ell$ hence with probability $1-p_{3,i}-p_{\geq \ell+1,i}$ 

If $d(u)=3$ then a Hyperaction of Type 2,3a,3b or 4 occurs. Assume that a Hyperaction of Type 3a occurs, that is  $u$ has 3 neighbors in $\G_i$, let them be $\{v,x_1,x_2\}$ 
and there is no edge from $x_1$ to $x_2$.  In this case 
\textsc{Reduce} proceeds and contracts $\{v,x_1,x_2\}$. The new vertex, say $z$,  has degree $d(x_1)+d(x_2)-2$. For the change in $ex_\ell$,   $\ell \geq 3$  we consider 3 cases. In all 3 cases
the max-edge removal results in the decrease of $ex_{\ell,i}$ by the amount of $\mathbb{I}(d(v)>\ell)$.

$\bullet$ \textbf{Case 3:} $d(x_1)=d(x_2)=3$. The new vertex has degree degree 4. Thus it does not contribute to $ex_{\ell,i}$  and  $ex_{\ell,i+1}-ex_{\ell,i}=-\mathbb{I}(ex_{\ell,i}>0)$.
Case 3 occurs with probability $p_{3,i}^3$

$\bullet$ \textbf{Case 4:} $d(x_1)=3, d(x_2)\neq 3$.
The new vertex $v_c$ contributes to $ex_{\ell,i+1}$ by the amount of $d(v_c)-\ell=[d(x_2)+3-2]-\ell= [d(x_2)+1]-\ell$ only if 
$d(v_c)=d(x_2)+1> \ell$ i.e. only if $d(x_2)\geq \ell$. If $d(x_2)>\ell$ then $x_2$ contributes to $ex_{\ell,i}$ by the amount of $d(x_2) -\ell$. 
\begin{align*}
ex_{\ell,i+1}-ex_{\ell,i}&=-\mathbb{I}(d(v)>\ell) + [d(x_2)+1-\ell]\mathbb{I}(d(x_2)+1> \ell) - [d(x_2)-\ell]\mathbb{I}(d(x_2)> \ell) 
\\&=-1 \text{ or } 0.
\end{align*}
 Case 4 occurs  with probability at most $2p_{3,i}^2(1-p_{3,i})$. 
 
$\bullet$ \textbf{Case 5:} $d(x_1)=j_1>3, d(x_2)=j_2>3$.
$x_h$ contributes  to $ex_{\ell,i}$ by the amount of $(j_h-\ell)\bI(j_h>\ell)$ for $h=1,2$, 
while the new vertex has degree $j_1+j_2-2$ and contributes by the amount of 
$$[(j_1+j_2-2)-\ell]\bI[j_1+j_2-2>\ell].$$
Therefore,
\begin{align}
ex_{\ell,i+1}-ex_{\ell,i}& = -\mathbb{I}(d(v)>\ell)
+ [(j_1+j_2-2)-\ell]\bI[j_1+j_2-2>\ell] \label{type3a} 
\\ &-(j_1-\ell)\bI(j_1>\ell)-(j_2-\ell)\bI(j_2>\ell)] \nonumber
\\&\leq \ell-2-\bI(d(v)>\ell)=\ell-3+\bI(ex_{\ell,i}=0) . \label{type3} 
\end{align}
In the last line we used $d(v)=\Delta_i$ implies that $\Delta_i>\ell$ iff $ex_{\ell,i}>0$.
Therefore $ \mathbb{I}(d(v)>\ell)= \mathbb{I}(ex_{\ell,i}>\ell)= 1-\mathbb{I}(ex_{\ell,i}=0).$ The first equality follows from the observation that  
 \eqref{type3a} implies that 
$ex_{\ell,i+1}-ex_{\ell,i}\geq -1$. Also from 
\eqref{type3} we can conclude that if $ex_{\ell,i}>0$ then 
\begin{align}\label{type3b}
ex_{\ell,i+1}-ex_{\ell,i}\leq \ell-3.
\end{align}
Case 5 occurs with probability $p_{3,i}\cdot p_{j_1,i} \cdot p_{j_2,i}$ (recall $j_1=d(x_1)>3, j_2=d(x_2)>3$).

$\bullet$ \textbf{Case 6:}A Hyperaction of Type 3b occurs. Then  $\{u,x_1,x_2\}$ is contracted into $v_c$ that satisfies $3\leq d(v_c) = d(x_1)+d(x_2)-4$. The rest of the analysis is similar to Cases 3,4 and 5. 

$\bullet$ \textbf{Case 7:} A  Hyperaction of Type 4 occurs. That is the edge removal of $\{v,u\}$ is followed by the contraction of $\{u,x_1,x_2\}$ to $v_c$ and the contraction of $\{v_c,z_1,z_2\}$ to $v_c'$.
The  new vertex has degree at most $d(z_1)+d(z_2)-2$. $d(x_1)=d(x_2)=3$ thus $x_1,x_2$ do not contribute to $ex_{\ell,i}$. Hence the same analysis as in cases 3,4,5 (with $z_2,z_2$ in place of $x_1,x_2$) applies.

$\bullet$ \textbf{Case 8}) A Hyperaction of Type 2 occurs. First the degree of $v$ is reduced by 1
and then $\{v,u,z\}$ is contracted (here $z$ is the second neighbor of $u$). The new vertex has degree $d(v)+d(z)-3$. Hence the same analysis as in cases 3,4,5 (with $v,z$ in place of $x_1,x_2$, where $v$ had degree $d(v)-1$) applies.

We summarize the above case analysis in Table 3.  
\begin{table}[h]
\centering
    \begin{tabular}{| c | c | c | c | c |}
    \hline
  \textbf{Case} &${d(u)}$ & \textbf{Hyperaction  that} & $ex_{\ell,i+1}-ex_{\ell,i}$ & \textbf{probability} 
 \\ & &   \textbf{takes place}& &\textbf{occurring}
     \\ \hline
 Case 1 &  $ \ell+1 \leq d(v)$ & Type 1 & -2 & $p_{\geq \ell+1,i}$ \\ \hline
 Case 2 & $ 3< d(u) \leq \ell$ & Type 1 &$-\mathbb{I}(ex_{\ell,i}>0)$& $1-p_{3,i}-p_{\geq \ell+1}$ \\ \hline
 Case 3 & 3 & Type 3a &-1& $p_{3,i}^3$ \\ \hline
 Case 4 & 3 & Type 3a &-1 or 0&  $2p_{3,i}(1-p_{3,i})$ \\ \hline
 Case 5 & 3 & Type 3a & $\in [-1,u_\ell]$& $p_{3,i}\sum_{j_1,j_2>3}p_{j_1,i}p_{j_2,i}$ \\ \hline
Case 6/7/8 &  & Type 3b/4/2 & $\in [-1,u_\ell]$& $o(n^{-0.75})$ \\ \hline
     \end{tabular}
    \caption{Case analysis for $3\leq \ell \leq k$, $u_\ell=\ell-3+\mathbb{I}(ex_{\ell,i}=0)$}
\end{table}
\\Therefore \eqref{bddex} is satisfied. In addition if $ex_{\ell,i}>0$  then $-\mathbb{I}(ex_{\ell,i}>0)=-1$ and $u_\ell=\ell-3$. Thus,
\begin{align*}
\e(ex_{\ell,i+1}-ex_{\ell,i}| \G_i) &\leq - 2p_{\ell+1,i}- (1-p_{3,i}-p_{\ell+1,i})- 
p_{3,i}^3
\\&+(\ell-3)\sum_{j_1,j_2>3}p_{j_1,i}p_{j_2,i}+(\ell-3)o(n^{-0.75})
\\& \leq - 1-p_{\ell+1,i}- 
p_{3,i}^3+(\ell-3) p_{3,i}(1-p_{3,i})^2+n^{-0.75}.
\end{align*}
\end{proof}
In Lemma \ref{general} and Corolary \ref{stopF}, using \eqref{bddex} and \eqref{elldrift}, we show that $F_{k,t^*}$ occurs w.h.p.\@ for various stopping times $t^*$. Hence w.h.p.\@ for $i<t^*$ the $i^{th}$ Hyperaction is good.
\begin{lem}\label{general}
Let $\G_0=G$ be a random (multi)-graph with degree sequence $\bd$, $n$ vertices, minimum degree 3 and  no loops that satisfies $ex_{k,0}=0$. 
Let $t^*$ be a stopping time such that the inequalities 
$i\leq t^*$ and $0<ex_{k,i}\leq \log^2 n$ imply 
\begin{align}\label{condition}
\e(ex_{k,i+1}-ex_{k,i}|\G_i,Q_{k}(G))<C \hspace{10mm} \text{ and } \hspace{10mm} e_i \geq n^{0.9}
\end{align}
 for some constant $C<0$. Then w.h.p.\@  
the event $F_{k,t^*}(G)$ occurs.
\end{lem}
\begin{proof}
Recall $F_{k,t^*}(G)$ is the event that 
for every $0\leq i < t^*$ 
\begin{itemize}
\item[] i) $ex_{k,i} \leq  \log^2 n$, 
\vspace{-2mm}
\item[] ii) \textsc{Reduce} applies a good Hyperaction to $\G_i$
\vspace{-2mm}
\item[] iii) there exists $z$ satisfying  $i-\log^2 n/(k-2)\leq z \leq i$ and  $ex_{k,z}=0$.
\end{itemize}
$ex_{k,0}=0$, thus conditioned on $Q_k(G)$ occurring,  \eqref{bddex} implies that if $ F_{k,t^*}(G)$ does not occur then 
there exists $ j \leq t^*-\log^2 n/(k-2)$  such that:
\begin{itemize}
\item[] i) $0 < ex_{k,j} \leq  (k-2)$,
\vspace{-2mm}  
\item[] ii) $0<ex_{k,j+h }<\log^2n$ for $ 0\leq h \leq \log^2 n/(k-2)-1 $ and 
\vspace{-2mm}
\item[] iii)$ex_{k, j+\log^2 n/(k-2) }>0$.
\end{itemize}
Indeed conditioned on $Q_k(G)$ occurring, \eqref{bddex} implies that for every such pair $j,h$  we have $$ex_{k,j+h} \leq ex_{k,j} +h(k-2) \leq \log^2 n.$$
Thus the inequality $\e(ex_{k,i+1}-ex_{k,i}|\G_i, Q_k(G))<C$ holds for $i\leq j+ \log^2 n/(k-2)-1$. \eqref{ed1} implies  $t^*\leq \tau \leq kn/2$. Thus
 Azuma-Hoeffding inequality (see Lemma \ref{ah}) gives
\begin{align*}
\pr( F_{k,t^*}(G) \text{ does not occur })& 
\leq  2 
\exp\bigg\{ \frac{-[(k-2)-C\cdot (\frac{\log^2 n}{k-2}-1) ]^2}{ 2\cdot   \frac{\log^2 n}{k-2} \cdot (k-2)^2}  \bigg\}+\pr(\neg Q_k(G))\\&
=o(n^{-0.5}).
\end{align*}
\end{proof}
\begin{cor}\label{stopF}
 a) For $k\in \{5,6,7 \}$ let $\G_0=G$ be a random graph with degree sequence $\bd$, minimum degree 3, maximum degree $k$ and no loops. Then  w.h.p.\@ $F_{k,\tau_{k-1}}$ occurs.
\\ b) For $8\leq k $ let $\G_0=G \in \cC_{3,k}$ be a random graph with degree sequence $\bd$, minimum degree 3, maximum degree $k$ and no loops. Then w.h.p.\@  $F_{k,t_{k-1}}$ occurs.
\end{cor}
\begin{proof}
It suffices to show that \eqref{condition} is satisfied. By setting $\ell=k$, \eqref{elldrift} implies
\begin{align}\label{numer1}
 \E{ex_{k,i+1}-ex_{k,i}| \G_i, Q_{k}(G)} 
\leq  -(1-p_{3,i}) -p_{3,i}^3 +(k-3) p_{3,i} (1-p_{3,i})^2+n^{-0.75}.
\end{align}
a) Maximizing \eqref{numer1} over $k\in \{5,6,7\}$ and $p_{3,i}\in [0,1]$ yields a maximum of $-0.08791$, attained at 
$k=7,p_{3,i}=0.40457$. 
\vspace{3mm}
\\b) $i<t_{k-1}$ implies that Lemma \ref{remA} is applicable and
hence that $p_{3,i} \leq 3/\sum_{j=3}^{k-1}\ca^{j-3}j +o(1)\leq 1/(k-2)+o(1)$. 
Maximizing \eqref{numer1} over $8\leq k$ and $p_{3,i} \leq 1/(k-2)$
yields a maximum  of $ -[(k-3)(k-2)^2-(k-3)^3+1]/(k-2)^3 +n^{-0.75}$ attained at $p_{3,i}=-1/(k-2)$. 
\end{proof}
\section{Proof of Lemma \ref{ind}}\label{sind}
We split the proof of Lemma \ref{ind} into 3 Lemmas.
The first one, Lemma \ref{halfb} is a slightly stronger version of Lemma \ref{propb}. It implies that for determining $t_{k-1}$ out of the  events $\cD_{k-1,j,i}$, $4\leq j \leq k-1$, $i\leq \tau $ it suffices to consider only the events $\cD_{k-1,k-1,i}$, $i\leq \tau$. Observe that $i<t_{k-1}$ implies that $n_{r,i}-\ca n_{r-i,i}>-[\log^2 n -(k-1)]n^{0.8}/2^r$ for $4\leq r \leq k-2$
The proof of Lemma \ref{halfb} is based on the fact that if $n_{r,i}-\ca n_{r-1,i}$ is 
close to $[\log^2 n -(k-1)]n^{0.8}/2^r$ then after the $i^{th}$ Hyperaction it will increase in expectation for $i<t_{k-1}$.
We then let $$t^*_{k-1}=\min\{t_{k-1}, \tau_{k-1}\}.$$
In Lemma \ref{halfb2}, using similar arguments as in Lemma \ref{halfb}, we show that
$\cD_{k-1,k-1,i}$ occurs for $i\leq t^*$. Lemmas \ref{halfb} and \ref{halfb2} together with part (b) of Corollary \ref{stopF} imply parts (i) and (ii) of Lemma \ref{ind}.

To prove part (ii), done in Lemma \ref{time} we first argue that $t^*_{k-1}< 1.5n_{k,0}+n^{0.6}.$ Then we use \eqref{edgedrift} to bound $e_{i+1}-e_i$ in terms of $p_{3,i}$. An upper bound on $p_{3,i}$ is provided by Lemma \ref{remA}.

Recall $\ca=1.17$.

\begin{lem}\label{halfb}
Let $8\leq k=O(1)$. Let $\G_0=G\in \cC_{3,k}\subseteq \cC_{3,k-1}$ be a random (multi)-graph with degree sequence $\bd$, maximum degree $k$, minimum degree 3 and no loops.
Then w.h.p.\@  for $4\leq r \leq k-2$  then event $\cD_{k-1,r, t_{k-1}}$ holds.
\end{lem}
\begin{proof}
Fix $r$, $4\leq r \leq k-2$.
We condition on the event  $F_{k,t_{k-1}}$ occurring.
Corollary \ref{stopF}   states that it  occurs w.h.p.\@. 
Hence for every $0\leq i < t_{k-1}$ the $i^{th}$ Hyperaction is good. Also  recall $\G_i \in \cC_{3,k-1}\supseteq \cC_{3,k}$ and $e_i\geq n^{0.9}$  for $i<t_{k-1}$. 
\vspace{3mm}
\\For $4\leq r\leq k-2$ if the event $\neg \cD_{k-1,r,t_{k-1}}$ occurs then
$$n_{r,t_{k-1}}-\ca n_{r-1,t_{k-1}} < - [\log^2 n-(k-1)]n^{0.8}/2^r$$
and
\begin{align}\label{51}
n_{r+1,i}-\ca n_{r,i} \geq  - [\log^2 n-(k-1)]n^{0.8}/2^{r+1} \text{ for } i < t_{k-1}.
\end{align}
For $i< t_{k-1}$ let $X_{r,i}= n_{r,i}-\ca n_{r,i}$.
\eqref{bddv} implies
\begin{align}\label{52}
|X_{r,i+1}-X_{r,i}| \leq 12
\end{align}
\eqref{52},  $\G_0 \in \cC_{3,k}$ and $\G_i \in \cC_{3,k-1}$ for $i<t_{k-1}$ imply that  for $t_{k-1}- n^{0.8}/(12\cdot 2^{r}) \leq i\leq t_{k-1}-1$ we have
\begin{align*}
 n_{r,i}-\ca n_{r-1,i} &\leq n_{r,t_{k-1}}-\ca n_{r-1,t_{k-1}}+12(t_{k-1}-i) 
\\& \leq  -\frac{ [\log^2 n-(k-1)]n^{0.8}}{2^r} +12 \cdot  \frac{n^{0.8}}{12\cdot 2^{r}}  
\leq   -\frac{[\log^2 n-k]n^{0.8}}{2^r}.
\end{align*}
Thus
\begin{align}\label{53}
n_{r-1,i} \geq \bigg(n_{r,i}+\frac{[\log^2 n-k]n^{0.8}}{2^r}\bigg) \bigg/ \ca. 
\end{align}
Using  \eqref{ev}, the following holds:
\begin{align}
\e(&X_{r,i+1}-X_{r,i}|\G_i) = \e(n_{r,i+1}-n_{r,i}|\G_i) - \ca \e(n_{r-1,i+1}-n_{r-1,i}|\G_i)  \nonumber
\\& =p_{r+1,i}-p_{r,i} 
 +p_{3,i}\bigg[\sum_{j_1+j_2-2=r} p_{j_1,i}p_{j_2,i}       
-2p_{r,i} \bigg]\nonumber 
\\&- \ca p_{r,i}+ \ca p_{r-1,i} 
 -\ca p_{3,i}\bigg[\sum_{j_1+j_2-2=r-1} p_{j_1,i}p_{j_2,i}       
-2p_{r-1,i} \bigg] 
 -o(n^{-0.75}) \nonumber
\\& \geq p_{r+1,i}-(1+\ca +2p_{3,i})p_{r,i} +(\ca+2\ca p_{3,1})p_{r-1,i}  
 -o(n^{-0.75}) -o(p_{r,i})-n^{-0.25} \label{111}
 \\& =  \frac{(r+1)n_{r+1,i}}{2e_i}- \frac{(1+\ca +2p_{3,i})rn_{r,i}}{2e_i} 
\nonumber 
 \\&+ \frac{(\ca+2\ca p_{3,1})(r-1)n_{r-1,i}}{2e_i} 
  -o(p_{r,i} ) -n^{-0.25} \label{eq2}
 \end{align}
 To derive \eqref{111} we used  
\begin{align*}
p_{3,i}&\sum_{j_1+j_2-2=r} p_{j_1,i}p_{j_2,i}       
 -\ca p_{3,i} \sum_{j_1+j_2-2=r-1} p_{j_1,i}p_{j_2,i}       
\\ &\geq  p_{3,i}\sum_{j_1+j_2-2=r} p_{j_1,i}p_{j_2,i}       
 -\ca p_{3,i} \sum_{j_1+j_2-2=r-1} p_{j_1,i}(p_{j_2+1,i}+n^{-0.09})/\ca   
 \\ &\geq - p_{3,i} \sum_{3\leq j_1 \leq r} p_{j_1,i}n^{-0.09}
 \geq - (p_{r,i} +n^{-0.09}) \sum_{3\leq j_1 \leq r} (p_{r,i} +n^{-0.09}) n^{-0.09}
 \geq -o(p_{r,i} ) - n^{-0.25}.
\end{align*}
In the second and third line of the above calculations we used that for $3\leq b_2 \leq b_1 \leq k-1$, $\G_i \in \cC_{3,k-1}$  implies
\begin{align*} 
p_{b_1,i}&=\frac{b_1n_{b_1,i}}{2e_i} \geq \frac{\ca^{b_1-b_2}b_1n_{b_2,i}-O(n^{0.8}\log^2n)}{2e_i} 
\\& \geq \frac{\ca^{b_1-b_2}b_1 b_2 n_{b_2,i}}{2b_2e_i}- O\bfrac{n^{0.8}\log^2n}{2n^{0.9}} \geq \ca^{b_1-b_2}p_{b_2,i}-n^{-0.09}.
\end{align*}
Using \eqref{51} and \eqref{53} in order to upper bound $n_{k+1,i}$ and $n_{k-1,i}$ respectively by terms involving only $n_{r,i}$, \eqref{eq2} implies
\begin{align} 
 \e(X_{r,i+1}-X_{r,i}|\G_i) 
& \geq   \frac{(r+1)[\ca n_{r,i}-[\log^2 n-(k-1)]n^{0.8}/2^{r+1}]}{2e_i}- \frac{(1+\ca +2p_{3,i})rn_{r,i}}{2e_i} \label{112}
\\&+ \frac{(\ca+2\ca p_{3,1})(r-1)[n_{r,i} + [\log^2 n-k]n^{0.8}/2^r  ]}{\ca \cdot 2e_i}
-o(p_{r,i} ) -n^{-0.25}
\nonumber
\\& =  \frac{(\ca-1 -2p_{3,i})n_{r,i}}{2e_i}  -o(p_{r,i} ) -n^{-0.25}  \nonumber
\\& + \frac{[2(r-1)(1+2p_{3,i})(\log^2 n-k) - (r+1)(\log^2 n-(k-1)) ]n^{0.8}}{2^{r+2}e_i} \nonumber
\\& \geq  \frac{(\ca -1  -2\cdot 0.081)n_{r,i}}{2e_i} -o(p_{r,i})
+ \frac{n^{0.8}\log n}{e_i} \label{113}
\\& \geq  \frac{0.005p_{r,i}}{r} 
+ \frac{n^{0.8}\log n}{e_i} -o(p_{r,i} )  \geq n^{-0.2}.
\label{114}
\end{align}
In \eqref{113} we used Lemma \ref{remA} i.e. $p_{3,i} \leq 0.081$.
In addition to \eqref{114}, if the event $\neg \cD_{k-1,r,t_{k-1}}$ occurs then
$ X_{r,t_{k-1}}< - [\log^2 n-(k-1)]n^{0.8}/2^r \leq X_{r,t_{k-1} -n^{0.8}{/11 \cdot 2^r}} $ and hence
\begin{align}\label{115}
 X_{r,t_{k-1}}- X_{r,t_{k-1} -n^{0.8}{/11 \cdot 2^r}}  = \sum_{j=t_{k-1} -n^{0.8}{/11 \cdot 2^r}}^{t_{k-1}-1}(X_{r,j+1}-X_{r,j})\leq 0.
 \end{align}
Using \eqref{52}, \eqref{114} and \eqref{115}, the
 Azuma-Hoeffding inequality gives us,
\begin{align*}
\pr( \neg \cD_{k-1,j, t_{k-1}}& \text{ for some } 3\leq j \leq r-2 )
\\& \leq \sum_{r=4}^{k-2} 
\pr\bigg(\sum_{j=t_{k-1} -n^{0.8}/12\cdot2^r}^{t_{k-1}-1}  X_{r,j+1}-X_{r,j} \leq0\bigg| F_{k,t_{k-1}} \bigg)+o(\neg \pr(F_{k,t_{k-1}}))
\\ &\leq  2   \exp \bigg\{-\frac{ 
(n^{-0.2} \cdot n^{0.8}/12\cdot 2^r)^2}{2\cdot 12^2 \cdot  n^{0.8}/12\cdot 2^r} \bigg\}  +o(n^{-0.5}) =o(n^{-0.5}).
\end{align*}
\end{proof}
Similar techniques, as the ones used in the proof of Lemma \ref{halfb}, are used in the proof of the  following Lemma. 
\begin{lem}\label{halfb2}
Let $8\leq k=O(1)$. Then w.h.p.
$$t^*_{k-1}=\tau_{k-1}<t_{k-1}.$$
\end{lem}
\begin{proof}
Given Lemmas \ref{halfb} 
it suffices to show that w.h.p.\@  the inequality
\begin{align}\label{mi}
n_{k-1,t^*_{k-1}}-\ca n_{k-2,t^*_{k-1}}\geq -[\log^2 n-(k-1)]n^{0.8}/2^{k-1} 
\end{align}
holds. Assume otherwise. Then \eqref{bddv} implies that for $t_{k-1}^*- n^{0.8}/(12\cdot 2^{k-1}) \leq i\leq t_{k-1}^*-1$ the following inequalities hold:
\begin{align}\label{new}
-[\log^2 n-k]n^{0.8}/2^{k-1} \geq  n_{k-1,i}-\ca n_{k-2,i} 
\geq - [\log^2 n-(k-1)]n^{0.8}/2^{k-1}.
\end{align}
Let $X_{k-1,i}=n_{k-1,i}-\ca n_{k-2,i}$.
Similarly to the derivation of \eqref{eq2} (by comparing \eqref{ev} and \eqref{ev2}) we have:
\begin{align}
\e(X_{k-1,i+1}-X_{k-1,i}|\G_i)& \geq -\frac{(1+\ca+2p_{3,i})(k-1)n_{k-1,i}}{2e_i}\nonumber
\\&+\frac{(\ca +2p_{3,i})(k-2)n_{k-2,i}}{2e_i} +\mathbb{I}(\Delta_i=k)-o(1)\nonumber
\\&\geq -\frac{(1+\ca +2p_{3,i}) (k-1)\cdot \ca n_{k-2,i}}{2e_i} \label{116}
\\&+\frac{(\ca+2p_{3,i})(k-2)n_{k-2,i}}{2e_i} +\mathbb{I}(\Delta_i=k)-o(1)
\nonumber
\\&= -\frac{ [\ca+\ca^2(k-1)+ (2\ca+2(\ca-1)(k-2)) p_{3,i} ]n_{k-2,i}}{2e_i} 
\nonumber
\\& +\mathbb{I}(\Delta_i=k)-o(1)
\nonumber
\\&= -\bigg(\frac{\ca^2+\ca+2\ca p_{3,i}}{k-2} +\ca^2+ 2(\ca-1) p_{3,i} \bigg)p_{k-2,i}
 +\mathbb{I}(\Delta_i=k)-o(1) \nonumber
 \\&\geq -2.1p_{k-2,i}  +\mathbb{I}(\Delta_i=k)-o(1).\nonumber
\end{align}
To derive \eqref{116} we used the LHS of \eqref{new}.
In the last line  we used $p_{3,i}\leq 0.081$ (Lemma \ref{remA}) and the inequality $k\geq 8$.
Let $t_\ell=t^*_{k-1}-n^{0.8}/(12\cdot 2^k), t_u=t^*_{k-1}-1$. Corollary  \ref{new1} and Lemma \ref{new2} imply
\begin{align}
\sum_{i=t_\ell}^{t_u} &\e(X_{k-1,i+1}-X_{k-1,i}|\G_i) 
\geq \sum_{i=t_\ell}^{t_u} [-2.1 p_{k-2,i}  +\mathbb{I}(\Delta_i=k)-o(1)] \nonumber
\\&\geq (t_u-t_\ell)[-2.1 p_{k-2,t_\ell} +0.999-(k-2)p_{3,t_\ell}-o(1)] \label{120}
 \\& \geq 0.01 (t_u-t_\ell). \label{121}
\end{align}
In the last inequality we used that $p_{e,t_\ell} \le 0.081$ and that $p_{k-2,t_\ell} \leq 0.47$. To derive the bound on $p_{k-2,t_\ell}$ observe that $\ca n_{k-2,t_\ell} \leq n_{k-1,t_\ell}+o(1)$ implies $\ca p_{k-2,t_\ell} \leq p_{k-1,t_\ell}+o(1)$.
 Therefore,  
\begin{align}
p_{k-2,t_\ell} &\leq (p_{k-2,t_\ell}+ p_{k-1,t_\ell})/(1+\ca) +o(1) 
\leq \bigg(1-\sum_{j=3}^{k-3} p_{j,t_\ell}\bigg)/(1+\ca)+o(1) \nonumber
\\&\leq \bigg(1-\sum_{j=3}^{k-3} \frac{j \cdot \ca^{j-3}p_{3,t_\ell}}{3} \bigg)/(1+\ca)+o(1). \label{122}
\end{align}
\eqref{120} is maximized over $0\leq p_{j,t_\ell}$ and \eqref{122} when 
$p_{3,t_\ell}=...=p_{k-3,t_\ell}=p_{k,t_\ell}=0$ and 
$p_{k-2,t_\ell}= 1/(1+\ca)+o(1) \leq 0.47$.

On the other hand  \eqref{mi}, \eqref{new} imply 
$$\sum_{j=t_\ell}^{t_u} X_{k-1,j+1}-X_{k-1,j}=X_{k-1,t_u}- X_{k-1,t_\ell} \leq 0.$$
Also \eqref{bddv} implies $|X_{k-1,i+1}-X_{k-1,i}| \leq 12$ 
Thus the Azuma-Hoeffding inequality gives us,
\begin{align*}
\pr(& t_{k-1} \leq \tau_{k-1}) \leq 
\pr\bigg(\sum_{j=t_l}^{t_{u}^*} X_{k-1,j+1}-X_{k-1,j} \leq 0 \bigg| F_{k,t_{k-1}} \bigg)+o(n^{-0.5})
\\ &\leq  2  \exp \bigg\{-\frac{( 0.01 n^{0.8}/12\cdot 2^{k-1})^2}{2\cdot 12^2 \cdot (n^{0.8}/12\cdot 2^{k-1})} \bigg\}  +o(n^{-0.5}) =o(n^{-0.5}).
\end{align*}
\end{proof}
The final part of Lemma \ref{ind} is proved in the following Lemma.
\begin{lem}\label{time}
Let $k\geq 8$.  
Then, w.h.p.
$$ \tau_{k-1} \leq 1.5n_{k,0}+n^{0.6} \text{ and } e_{\tau_{k-1}}\geq (1-4/k)e_0=\Omega(n).$$
\end{lem}
\begin{proof} 
We condition on the event $F_{\tau}$ occurring. Lemmas \ref{halfb} and \ref{halfb2} imply that w.h.p. $F_{k,\tau_{k-1}} \subseteq F_{k,t_{k-1}}$. Hence, from Corollary \ref{stopF} follows that it  occur w.h.p.\@. Using the bound provided by \eqref{elldrift} (with $\ell=k-1$) we get:
\begin{align*}
\E{ex_{k-1,i+1}-ex_{k-1,i}| \G_i}
&\leq  -(1-p_{3,i})-p_{k,i} -p_{3,i}^3 + (k-4)p_{3,i}+n^{-0.75}
\\& \leq -1+(k-4)p_{3,i} +n^{-0.75} \leq -0.67.
\end{align*}
In the last inequality we use that  $(k-4)p_{3,i}\leq (k-4)\cdot 3/ (\sum_{i=3}^{k-1} i\cdot \ca^{i-3}) +o(1) \leq 0.33$ for $k\geq 8$ (see Lemma \ref{remA}). 
Since $ex_{k-1,0}=n_{k,0}$ and $ex_{k-1,\tau_{k-1}}=n_{k,\tau_{k-1}}
=0$, using \eqref{bddex}, by the Azuma Heoffding inequality we get
\begin{align*}
\pr(\tau_{k-1} &\geq 1.5n_{k,0} +n^{0.6})\leq 
 \pr( ex_{k-1,i}>0 \text{ for } i\leq 1.5n_{k,0} +n^{0.6})
\\& \leq  2\exp\bigg\{ -\frac{[n_{k,0}-0.67 \cdot(1.5 n_{k,0} +n^{0.6})]^2}{2(k-2)^2(1.5n_{k,0}  +n^{0.6})}\bigg\} =o(n^{-0.5}).
\end{align*}
For $i\leq \tau_{k-1}<t_{k-1}$, using $p_{3,i}\leq0.081$ (Lemma \ref{remA}),
 \eqref{edgedrift} implies that $\e(e_{i+1}-e_i|\G_i) \geq -1.2$.
Conditioned on the event $\tau_{k-1}  \leq 1.5n_{k,0} +n^{0.6}$
, using \eqref{bdedges}, the Azuma Hoeffding inequality gives, 
\begin{align*}
\pr(e_0-e_{\tau_{k-1}} \geq 1.9n_{k,0} +n^{0.7})& \leq  
2\exp\bigg\{- \frac{[ 1.9n_{k,0} +n^{0.7}   -1.2( 1.5n_{k,0} +n^{0.6})]^2}{2\cdot 6^2 \cdot( 1.5n_{k,0} +n^{0.6}) } \bigg\}  \\&=o(n^{-0.5}).
\end{align*}
$k\geq 8$ implies,
\begin{align*}
e_{\tau_{k-1}} &\geq e_0 - 1.9n_{k,0} -n^{0.7}
\geq \sum_{j=3}^{k-1} 0.5jn_{j,0}+ 0.5kn_{k,0} -1.9n_{k,0} -n^{0.7}
\\&  \geq (1-4/k)\sum_{j=3}^{k-1} 0.5jn_{j,0}+ 0.5(1-4/k)kn_{k,0} =(1-4/k)e_0
\end{align*}
\end{proof}

\section{Proof of Lemma \ref{567}}\label{s567}
In this section,  we fix $k\in \{5,6,7\}$. We also let $\bd$ be a degree sequence with minimum degree 3 and maximum degree $k$ and $G=\G_0$ be a random graph with degree sequence $\bd$ and no loops. For the rest of this section we condition on $F_{k,\tau_{k-1}}$ occurring. Corollary \ref{stopF} states that it does occurs w.h.p.\@.  

The proof of Lemma \ref{567} can be split into two parts. In the first part, done as Lemma, \ref{part1} we let
$$t^*=\min\{i: ex_{k-1,i} \leq 10^{-2}e_{i}\}$$ 
and we show that  
$e_{t^*}\geq e_0/10^{25}=  \Omega(n)$.
In the second part, done as Lemma \ref{part2},
we show that $\tau_{k-1} \leq t^*+6 e_{t^*}/10^2.$

To  prove the first part, for $i<\tau$ we let 
$$X_i=\bigg[ (ex_{k-1,i+1}-ex_{k-1,i}) - 2.4  \frac{ ex_{k-1,i}}{e_i}  (e_{i+1}-e_{i}) \bigg]$$
Roughly speaking $X_i$ compares the rates of decrease of $ex_{k-1,i}$ and $e_i$ 
after the $i^{th}$ Hyperaction. In Lemma \ref{-drift} we show that $X_i$ decreases in expectation. Using this fact, we show that after a number of Hyperactions the ratio $ex_{k-1,i}/e_i$ decreases. As a consequence we will prove that there exists  
$t^*$ such that $e_{t^*}\geq e_0/10^{25}=  \Omega(n)$ and 
$ex_{k-1,t^*} \leq 10^{-2}e_{t^*}.$

For the second part it suffices to argue that $ex_{k,i}$ is decreased by at least 0.2 in expectation after the $i^{th}$ Hyperaction for $i\leq \tau_{k-1}$ (done in Lemma \ref{-drift}).
From the first part we have that $ex_{k-1,t^*} \leq e_{t^*}/100$ and therefore, in expectation, $ex_{k-1,t^*}$ reaches 0 in $5e_{t^*}/100$ Hyperactions. At the same time the number of edges is decreased by at most 6 (see \eqref{bdedges}) and hence after $6e_{t^*}/100$ Hyperactions it will remain linear in $n$.

We start with a technical Lemma. \eqref{-0.2} and \eqref{faster} are used in the proofs of Lemmas \ref{part2} and \ref{part1} respectively.
\begin{lem}\label{-drift}
For $i < \tau_{k-1}$ 
\begin{align}\label{-0.2}
\E{ex_{k-1,i+1}-ex_{k-1,i}|\G_i}\leq -0.2.
\end{align}
Furthermore, 
\begin{align} \label{faster}
\E{X_i|\G_i}<0 \vspace{10mm} \text{ and } \vspace{5mm} |X_i| \leq k+11.
\end{align}
\end{lem}  
\begin{proof}
$i<\tau_{k-1}$ implies  that $ex_{k-1,i}>0$. 
By setting $\ell=k-1$ in \eqref{elldrift} we get 
\begin{align*}
\E{ex_{k-1,i+1}-ex_{k-1,i}| \G_i }\leq  -(1-p_{3,i})-p_{3,i}^3 +(k-4) p_{3,i} (1-p_{3,i})^2+n^{-0.75} \leq -0.2.
\end{align*}
The last inequality can be easily verified numerically
since for $k\in \{5,6,7\}$. Its maximum over $k\in \{5,6,7\}, p_{3,i}\in [0,1]$ is $-0.23020$ and it is attained at $k=7, p_{3,i} = 0.42265$. 
\vspace{3mm}
\\$F_{k,\tau_{k-1}}$ occurs and 
therefore for $i < \tau_{k-1}$,
$$ex_{k-1,i}=n_{k,i}+ex_{k,i}= n_{k,i}+ O(\log^2 n).$$ 
Also $i<\tau_{k-1}$ implies that $e_i> n^{0.9}$. Thus, 
\begin{align}\label{aux1}
\frac{kex_{k-1,i}}{2e_{i}}= \frac{kn_{k,i}+O(\log^2 n)}{2e_{i}}=p_{k,i}+o(1).   
\end{align}
\eqref{edgedrift} and \eqref{aux1} imply 
\begin{align}\label{aux2}
\E{\frac{ ex_{k-1,i}}{e_i}  (e_{i+1}-e_{i})  \bigg|\G_i  }
&= \bigg(\frac{2p_{k,i}}{k}+o(1)\bigg) \cdot (-1-2p_{3,i}   + o(n^{-0.75})) \nonumber
\\ &= -\frac{ 2p_{k,i}(1 + 2p_{3,i})}{k}+o(1). 
\end{align}
 \eqref{elldrift} (with $\ell= k-1$) and \eqref{aux2} imply
\begin{align*}
k\E{X_i|\G_i}\leq  k[ -(1-p_{3,i})-p_{k,i} -p_{3,i}^3 +(k-4) p_{3,i} (1-p_{3,i})^2]     + 4.8 p_{k,i} (1 + 2p_{3,i}) +o(1)
\end{align*}
The maximum of the above expression over $k\in \{5,6,7\}$, $p_{3,i},p_{k,i},p_{3,i}+p_{k,i} \in [0,1]$ is $-391/1960$ attained at 
$k=5$, $p_{3,i}=99/196$ and $p_{5,i}=97/196$.
\eqref{bdedges} and \eqref{bddex} imply  $|e_{i+1}-e_i|\leq 6$ and 
 $|ex_{k-1,i+1}-ex_{k-1,i}| \leq k-4$ respectively.
Therefore $|X_i|\leq (k-4)+2.4 \cdot 1 \cdot 6 \leq k+11.$
\end{proof}

\begin{lem}\label{part1}
Let
$$t^*=\min\{i: ex_{k-1,i} \leq 10^{-2}e_{i}\}.$$ 
Then w.h.p. $e_{t^*}\geq e_0/10^{25}=  \Omega(n)$.
\end{lem}
\begin{proof}
We start by proving Claim 1. We later use Claim 1 to show that there exists $t^*<\tau_{k-1}$ such that $ex_{k-1,t^*}\leq e_{t^*}/10^2$ and $e_{t^*}\geq e_0/10^{24}=\Omega(n)$.
\vspace{3mm}
\\ \noindent{\bf Claim 1:} W.h.p.\@ for every $j\in \mathbb{N}$ such that $e_j=\Omega(n)$ and $j<\tau_{k-1}$ at least one of the following hold:
\begin{itemize}
\item[] i) there exists $s_j\geq j$ such that $e_{s_j}\geq e_j/10^3 =\Omega(n)$ and
$ex_{k-1,s_j}\leq e_{s_j}/10^2$
\vspace{-2mm}
\item[] ii) there exists $s_j \geq j$ such that $e_{s_j}\geq e_j/10^3 =\Omega(n)$ and ${ex_{k-1,s_j}}/e_{s_j} 
\leq 0.5{ex_{k-1,j}}/{e_j}$.
\end{itemize} 
\noindent{\bf Proof of Claim 1:}
Let $j\in \mathbb{N}$ be such that $e_j=\Omega(n)$. 
Let $$s^*:=\min\{i \geq j: ex_{k-1,i} \leq 0.11 ex_{k-1,j}\}.$$
Then for $j\leq i<s^*$,
$$e_i\geq ex_{k-1,i}\geq 0.11ex_{k-1,j} \geq 0.11e_j/10^2=\Omega(n).$$
Thus for $j\leq i<s^*$ the inequalities 
$e_i=\Omega(n)$, $ex_{k-1,i}>0$ hold. Therefore $i\leq s^* \leq\tau_{k-1}$. 
Lemma \ref{-drift}  implies that w.h.p.\@  
 $\e(X_i|\G_i)\leq 0$ and 
$|X_i|\leq  k+11$ for every 
$i < s^*$.
\eqref{ed1} implies that $s^*\leq \tau \leq kn/2 \leq 3.5n$.
Hence, from the Azuma-Hoeffding Inequality we have, 
\begin{align}\label{eq1}
\pr\bigg( \sum_{r=j}^{s^*-2} X_r > n^{0.6}\bigg) & \leq s^* \cdot 2\exp\bigg\{-\frac{(n^{0.6})^2}{ 2(k+11)^2 \cdot 3.5n}\bigg\} +\pr( \neg F_{k,\tau_{k-1}}) \nonumber
\\&\leq 7ne^{-n^{0.19}}+ o(n^{-0.5}) =o(n^{-0.5}).
\end{align} 
% \leq
% \pr\bigg(\sum_{r=j}^i X_r - \sum_{r=j}^i \e(X_r|\G_i) > n^{0.6}\bigg)
Now for $j\leq i<s^*$ let 
 $$Y_i=\big[ (ex_{k-1,i+1}-ex_{k-1,i}) - 1.2  \frac{ ex_{k-1,j}}{e_j}  (e_{i+1}-e_{i}) \big].$$
Assume that (ii) does not hold. Then for $j\leq i<s^*$, ${ex_{k-1,s_j}}/e_{s_j} 
> 0.5{ex_{k-1,j}}/{e_j}$. In the second case  
\eqref{eq1}, the definitions of $X_i,Y_i$  and  the fact that $e_i$ is decreasing with respect to $i$ imply that w.h.p.\@ 
$$n^{0.6} \geq \sum_{r=j}^{\tau^*-2}  X_i \geq \sum_{r=j}^{\tau^*-2}  Y_i.$$
Hence,
\begin{align*}
0.11 ex_{k-1,j} &\leq ex_{k-1,s^*-1}= ex_{k-1,j}-\sum_{i=j}^{s^*-2} (ex_{k-1,i}-ex_{k-1,i+1})
\\& \leq ex_{k-1,j}-  \sum_{i=j}^{s^*-2} [(ex_{k-1,i}-ex_{k-1,i+1})+Y_i] +n^{0.6}
\\& \leq ex_{k-1,j} + \sum_{i=j}^{s^*-2} 1.2 \frac{ex_{k-1,j}}{e_{j}}  (e_{i+1}-e_{i}) +n^{0.6}
\\& =ex_{k-1,j} +1.2 \frac{ex_{k-1,j}}{e_{j}}  (e_{s^*-1}-e_j) +n^{0.6}
\\& =-0.2 ex_{k-1,j} + 1.2 \frac{ex_{k-1,j}}{e_{j}}  e_{s^*-1} +n^{0.6}.
\end{align*}
The last equality implies that $0.31 ex_{k-1,j} \leq 1.2 \frac{ex_{k-1,j}}{e_{j}}  e_{s^*-1} +n^{0.6}$ and hence 
$$\frac{0.31e_j}{1.2} \leq e_{s^*-1} +  \frac{e_j}{1.2e_{k-1,j}} \cdot n^{0.6}.$$
Assume that  condition (i) does not hold. Then, $e_j< 10^2 ex_{k-1,j}$ and so 
\begin{align}\label{w1}
0.25 e_j \leq e_{s^*-1} + 10^2 n^{0.6}.
\end{align}
 Conditioned on $F_{k,\tau_{k-1}}$ the $(s^*-1)^{th}$ Hyperaction is good thus \eqref{bdedges} 
gives us
\begin{align}\label{w2}
e_{s^*}\geq e_{s^*-1}-  6 .
\end{align}
The definition of $s^*$, \eqref{w1} and \eqref{w2} imply
\begin{align}\label{11}
\frac{ex_{k-1,s^*}}{e_{s^*}} \leq  \frac{0.11 ex_{k-1,j}}{0.25e_{j}-10^2 n^{0.6}-6}\leq 0.5 \frac{ ex_{k-1,j}}{e_{j}} .
\end{align}
In the last equality we used $e_j = \Omega(n)$.
Finally from  \eqref{w1}, \eqref{w2} and \eqref{11} 
we get that (ii) holds. Thus either (i) or (ii) holds.
\qed
\vspace{3mm}
\\ By iterelively applying Claim 1 we get that w.h.p.\@
there exists a sequence $0=s_0,s_1,s_2,...,s_8$ such that 
\begin{itemize}
\item[]i) $ex_{k-1,s_i}/e_{s_i} \leq 0.5 ex_{k-1,s_{i-1}}/e_{s_{i-1}}$  or $ex_{k-1,s_i}\leq e_{s_i}/100$ for $i\leq 8$ and
\vspace{-2mm}
\item[]ii) $e_{s_i}\geq e_{s_{i-1}}/10^3$ for $i\leq 8$.
\end{itemize}
Let $t^*= \min\{s_i: ex_{k-1,s_i}\leq e_{s_i}/100\}$. If $t^*>s_8$ then  $ex_{k-1,s_i}/e_{s_i} \leq 0.5 ex_{k-1,s_{i-1}}/e_{s_{hi-1}}$ for $i\leq 8$ and so
$$ \frac{ex_{k-1,s_8}}{e_{s_8}} 
\leq 0.5\frac{ex_{k-1,s_7}}{e_{s_7}}
\leq 0.5^2 \frac{ex_{k-1,s_6}}{e_{s_6}} \leq \cdots 
\leq 0.5^8 \frac{ex_{k-1,s_0}}{e_{s_0}} = 0.5^7 \frac{ex_{k-1,s_0}}{2e_{s_0}} \leq 0.5^7 \leq 0.01. $$
Hence, $t^* \leq s_8$ and  $e_{t^*}\geq e_{s_8} \geq e_{0}/(10^3)^8=e_0/10^{24}$.
\end{proof}
Lemma \ref{567} follows from Lemma \ref{part2} and Corollary \ref{stopF}. Corollary \ref{stopF} states that  w.h.p. $F_{k,\tau_{k-1}}$ occurs and hence the first $\tau_{k-1}-1$ Hyperactions are good.
\begin{lem}\label{part2}
W.h.p. $\tau_{k-1} \leq t^*+6\cdot 10^{-6} e_{t^*}$ and  $e_{\tau_{k-1}}\geq e_0/2^{25}$.
\end{lem}
Let 
$Z(\G_{t^*})$ be the event that $ex_{k-1,j}>0$ for $j\leq t^*+6e_{t^*}/10^2$.
If $Z(\G_{t^*})$   occurs,  
conditioned on $F_{k,\tau_{k-1}}(G)$, 
the first $t^*+6 e_{t^*}/10^2$ Hyperactions are good.
Thus for $t^*\leq j\leq t^*+6\cdot 10^{-6} e_{t^*}$
the inequality $|e_{j+1}-e_{j}| \leq 6$ holds (see \eqref{bddex}) 
which implies 
\begin{align}\label{final}
e_{t^*+i} \geq e_{t^*}-6i \geq 0.5 e_{t^*} \geq e_0/10^{25} =\Omega(n) \text{ for }i\leq 6e_{t^*}/10^2.
\end{align}
Moreover  \eqref{-0.2} and \eqref{bddex} state
$$\e(ex_{k-1,j+1}-ex_{k-1,j}|\G_j) =-0.2 \hspace{5mm} \text{ and } \hspace{5mm} 
|ex_{k-1,j+1}-ex_{k-1,j}|\leq k-4.$$
In addition 
$$\sum_{j=t^*}^{t^*+6\cdot e_{t^*}/10^2} \e(ex_{k-1,j}-ex_{j}|\G_j)+ex_{k-1,t^*} \leq -0.2\cdot 6 e_{t^*}/10^2 + e_{t^*}/10^2
 <-0.2 e_{t^*}/10^2.$$  
Therefore, since $e_{t^*}=\Omega(n)$, the  Azuma-Hoeffding inequality (see Lemma \ref{ah}) implies  
\begin{align*}
\pr(Z(\G_{t^*}))
&\leq 2\exp
\bigg\{-\frac{(0.2 e_{t^*}/10^2)^2}{2 [6 e_{t^*}/10^2]\cdot  (k-4)^2}  \bigg\}+o(n^{-0.5}) =o(n^{-0.5}). 
\end{align*}
Hence $\tau_{k-1} \leq  t^*+6\cdot 10^{-6} e_{t^*}$. 
\eqref{final} implies that $e_{\tau_{k-1}} \leq e_0/10^{25}$.
\qed
\section{Conclusion}
In this paper we have analyzed a variant of a Karp-Sipser algorithm and we have shown that w.h.p.\@ it finds a perfect matching in random $k=O(1)$-regular graphs.
We have demonstrated that if the initial graph is a random graph with a given degree sequence that   has some ``nice" properties then those properties are it retained  throughout the execution of \textsc{Reduce}, a key element used in the verification of the great efficiency of the algorithm.
It is natural to try to extend this approach to prove the correctness of the algorithm for $G_{n,p}$, as originally intended \cite{KS}.

\appendix

\section{Appendix A: Diagrams of Hyperactions of interest}
{\bf Type 2.}

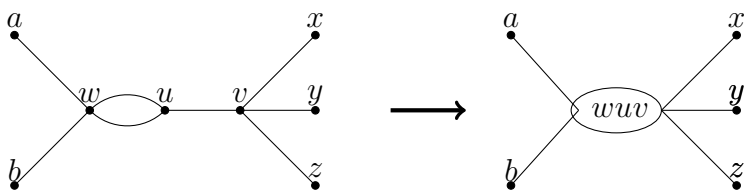
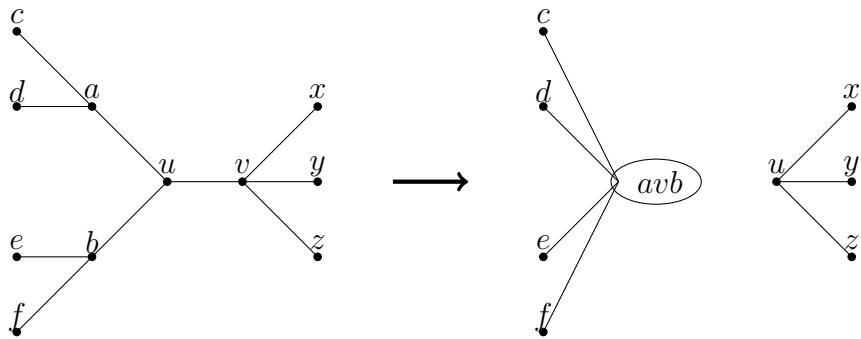
\begin{center}
\pic{
\node at (0,0.2) {$w$};
\node at (1,0.2) {$u$};
\node at (2,0.2) {$v$};
\node at (3,1.2) {$x$};
\node at (3,0.2) {$y$};
\node at (3,-.8) {$z$};
\node at (-1,1.2) {$a$};
\node at (-1,-0.8) {$b$};
\draw (1,0) -- (2,0);
\draw (2,0) -- (3,1);
\draw (2,0) -- (3,-1);
\draw (2,0) -- (3,0);
\draw (0,0) -- (-1,1);
\draw (0,0) -- (-1,-1);
\draw (0,0) to [out=45,in=135] (1,0);
\draw (0,0) to [out=-45,in=-135] (1,0);
\draw [fill=black] (0,0) circle [radius=.05];
\draw [fill=black] (1,0) circle [radius=.05];
\draw [fill=black] (2,0) circle [radius=.05];
\draw [fill=black] (3,0) circle [radius=.05];
\draw [fill=black] (3,1) circle [radius=.05];
\draw [fill=black] (3,-1) circle [radius=.05];
\draw [fill=black] (-1,1) circle [radius=.05];
\draw [fill=black] (-1,-1) circle [radius=.05];
\draw [->] [ultra thick] (4,0) -- (5,0);
\draw (7,0) ellipse (.6 and .3);
\node at (7.05,0) {$wuv$};
\draw (7.6,0) -- (8.6,1);
\draw (7.6,0) -- (8.6,-1);
\draw (7.6,0) -- (8.6,0);
\draw (5.6,1) -- (6.5,0);
\draw (5.6,-1) -- (6.5,0);
\draw [fill=black] (8.6,1) circle [radius=.05];
\draw [fill=black] (8.6,-1) circle [radius=.05];
\draw [fill=black] (8.6,0) circle [radius=.05];
\draw [fill=black] (5.6,-1) circle [radius=.05];
\draw [fill=black] (5.6,1) circle [radius=.05];
\node at (5.6,1.2) {$a$};
\node at (5.6,-0.8) {$b$};
\node at (8.6,1.2) {$x$};
\node at (8.6,0.2) {$y$};
\node at (8.6,-.8) {$z$};
\node at (8.6,0.2) {$y$};
\node at (8.6,-.8) {$z$};
}
\end{center}

{\bf Type 3.}

\begin{center}
\pic{
\node at (1,0.2) {$u$};
\node at (2,0.2) {$v$};
\node at (3,1.2) {$x$};
\node at (3,0.2) {$y$};
\node at (3,-0.8) {$z$};
\node at (0,1.2) {$a$};
\node at (0,-.8) {$b$};
\node at (-1,2.2) {$c$};
\node at (-1,1.2) {$d$};
\node at (-1,-.8) {$e$};
\node at (-1,-1.8) {$f$};
\draw (1,0) -- (2,0);
\draw (2,0) -- (3,1);
\draw (2,0) -- (3,-1);
\draw (2,0) -- (3,0);
\draw (0,1) -- (1,0);
\draw (0,1) -- (-1,2);
\draw (0,1) -- (-1,1);
\draw (0,-1) -- (-1,-1);
\draw (0,-1) -- (-1,-2);
\draw (0,-1) -- (1,0);
\draw [fill=black] (-1,2) circle [radius=.05];
\draw [fill=black] (-1,1) circle [radius=.05];
\draw [fill=black] (-1,-2) circle [radius=.05];
\draw [fill=black] (-1,-1) circle [radius=.05];
\draw [fill=black] (0,1) circle [radius=.05];
\draw [fill=black] (0,-1) circle [radius=.05];
\draw [fill=black] (1,0) circle [radius=.05];
\draw [fill=black] (2,0) circle [radius=.05];
\draw [fill=black] (3,0) circle [radius=.05];
\draw [fill=black] (3,1) circle [radius=.05];
\draw [fill=black] (3,-1) circle [radius=.05];
\draw [->] [ultra thick] (4,0) -- (5,0);
\draw (7.5,0) ellipse (.6 and .3);
\node at (7.55,0) {$avb$};
\draw (9.1,0) -- (10.1,1);
\draw (9.1,0) -- (10.1,-1);
\draw (9.1,0) -- (10.1,0);
\draw [fill=black] (9.1,0) circle [radius=.05];
\draw [fill=black] (10.1,1) circle [radius=.05];
\draw [fill=black] (10.1,-1) circle [radius=.05];
\draw [fill=black] (10.1,0) circle [radius=.05];
\node at (9.1,0.2) {$u$};
\node at (10.1,1.2) {$x$};
\node at (10.1,0.2) {$y$};
\node at (10.1,-.8) {$z$};
\node at (6,2.2) {$c$};
\node at (6,1.2) {$d$};
\node at (6,-.8) {$e$};
\node at (6,-1.8) {$f$};
\draw [fill=black] (6,2) circle [radius=.05];
\draw [fill=black] (6,1) circle [radius=.05];
\draw [fill=black] (6,-1) circle [radius=.05];
\draw [fill=black] (6,-2) circle [radius=.05];
\draw (6,2) -- (7,0);
\draw (6,1) -- (7,0);
\draw (6,-2) -- (7,0);
\draw (6,-1) -- (7,0);
}
\end{center}
We allow the edge $\{a,b\}$ to be a single edge in this construction. This gives us a Type 3b Hyperaction.
{\bf Type 4.}

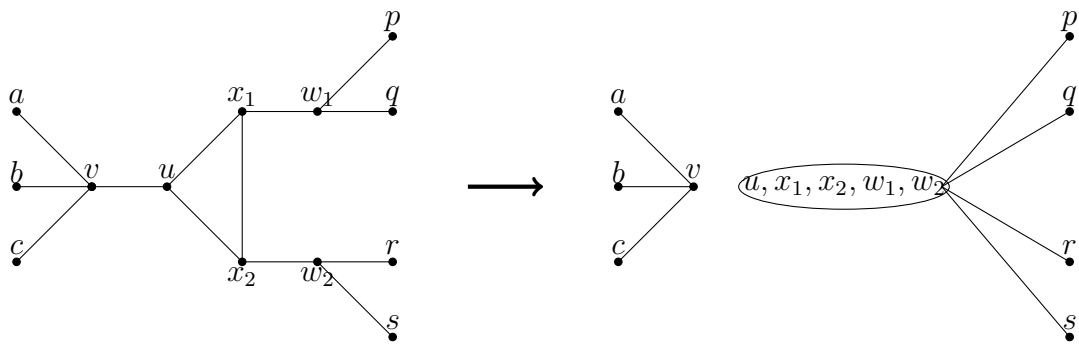
\begin{center}
\pic{
\node at (0,0.2) {$v$};
\draw [fill=black] (0,0) circle [radius=.05];
\node at (-1,1.2) {$a$};
\node at (-1,0.2) {$b$};
\node at (-1,-0.8) {$c$};
\draw [fill=black] (-1,1) circle [radius=.05];
\draw [fill=black] (-1,0) circle [radius=.05];
\draw [fill=black] (-1,-1) circle [radius=.05];
\draw (-1,1) -- (0,0);
\draw (-1,0) -- (0,0);
\draw (-1,-1) -- (0,0);
\draw [fill=black] (1,0) circle [radius=.05];
\node at (1,0.2) {$u$};
\draw (0,0) -- (1,0);
\node at (2,1.2) {$x_1$};
\node at (2,-1.2) {$x_2$};
\draw [fill=black] (2,1) circle [radius=.05];
\draw [fill=black] (2,-1) circle [radius=.05];
\draw (1,0) -- (2,1);
\draw (1,0) -- (2,-1);
\draw (2,-1) -- (2,1);
\node at (3,1.2) {$w_1$};
\node at (3,-1.2) {$w_2$};
\draw [fill=black] (3,1) circle [radius=.05];
\draw [fill=black] (3,-1) circle [radius=.05];
\draw (2,1) -- (3,1);
\draw (2,-1) -- (3,-1);
\node at (4,2.2) {$p$};
\node at (4,1.2) {$q$};
\node at (4,-0.8) {$r$};
\node at (4,-1.8) {$s$};
\draw [fill=black] (4,2) circle [radius=.05];
\draw [fill=black] (4,1) circle [radius=.05];
\draw [fill=black] (4,-1) circle [radius=.05];
\draw [fill=black] (4,-2) circle [radius=.05];
\draw (3,1) -- (4,2);
\draw (3,1) -- (4,1);
\draw (3,-1) -- (4,-2);
\draw (3,-1) -- (4,-1);
\draw [->] [ultra thick] (5,0) -- (6,0);
\node at (8,0.2) {$v$};
\draw [fill=black] (0,0) circle [radius=.05];
\node at (7,1.2) {$a$};
\node at (7,0.2) {$b$};
\node at (7,-0.8) {$c$};
\draw [fill=black] (7,1) circle [radius=.05];
\draw [fill=black] (7,0) circle [radius=.05];
\draw [fill=black] (7,-1) circle [radius=.05];
\draw [fill=black] (8,0) circle [radius=.05];
\draw (8,0) -- (7,1);
\draw (8,0) -- (7,0);
\draw (8,0) -- (7,-1);
\draw (10,0) ellipse (1.4 and .3);
\node at (10,0) {$u,x_1,x_2,w_1,w_2$};
\node at (13,2.2) {$p$};
\node at (13,1.2) {$q$};
\node at (13,-0.8) {$r$};
\node at (13,-1.8) {$s$};
\draw [fill=black] (13,2) circle [radius=.05];
\draw [fill=black] (13,1) circle [radius=.05];
\draw [fill=black] (13,-1) circle [radius=.05];
\draw [fill=black] (13,-2) circle [radius=.05];
\draw (11.3,0) -- (13,2);
\draw (11.3,0) -- (13,1);
\draw (11.3,0) -- (13,-1);
\draw (11.3,0) -- (13,-2);
}
\end{center}

\section{Appendix B: Proof of Lemma \ref{remA}}

\begin{proof}
\noindent $\G_i \in \cC_{3,k-1}$ implies that $n_{j,i}\geq \ca^{j-3}n_{3,i} -o(n^{0.85})$ for $3 \leq j \leq k-1$.
Therefore 
\begin{align}\label{131}
p_{3,i} =\frac{3n_{3,i}}{2e_i} \leq \frac{3n_{3,i}}{\sum_{j=3}^{k-1}jn_{j,i}+n^{0.85})}+o(1) 
\leq  \frac{3}{\sum_{j=3}^{k-1} \ca^{3-j}j}+o(1).
\end{align}
Finally for $k \geq 8$ \eqref{131} implies
\begin{align*}
p_{3,i} =\frac{3n_{3,i}}{2e_i}  
\leq  \frac{3}{\sum_{j=3}^{k-1} \ca^{3-j}j}+o(1)
\leq  \frac{3}{\sum_{j=3}^{7} \ca^{3-j}j}+o(1) \leq 0.081.
\end{align*}
\end{proof}

\end{document}